\documentclass[11pt]{amsart}

\usepackage{amsmath, amssymb,  mathabx, amsfonts,enumerate}
\usepackage[all]{xy}
\usepackage{amscd,stmaryrd}
\usepackage{comment}
\usepackage{mathtools}
\usepackage{tikz} 
\usepackage{tikz-cd} 
\usepackage{lipsum}
\tikzcdset{diagrams={nodes={inner sep=7.5pt}}}

\usepackage{url}
\usepackage{hyperref}

\newtheorem{theorem}{Theorem}[section]
\newtheorem{lemma}[theorem]{Lemma}
\newtheorem{corollary}[theorem]{Corollary}

\newtheorem{theoremletter}{Theorem}

 \theoremstyle{definition}
 \newtheorem{definition}[theorem]{Definition}

 \newtheorem{example}[theorem]{Example}

  \newtheorem*{example*}{Example}

\numberwithin{equation}{section}
 
\newcommand {\N}{\mathbb{N}} 
\newcommand {\Z}{\mathbb{Z}}

\newcommand{\G}{\mathbb{G}}

\newcommand{\LL}{\mathcal{L}}

\DeclareMathOperator{\Ker}{Ker}
\DeclareMathOperator{\End}{End}

\DeclareMathOperator{\im}{Im}

\DeclareMathOperator{\Id}{Id}

\begin{document}
\title[Pointwise and decidable properties of NUCA]{Some pointwise and decidable properties of non-uniform cellular automata}    
\author[Xuan Kien Phung]{Xuan Kien Phung}
\address{Département d'informatique et de recherche opérationnelle,  Université de Montréal, Montréal, Québec, H3T 1J4, Canada.}
\email{phungxuankien1@gmail.com}   
\subjclass[2020]{05C25, 20F69, 37B10, 37B15, 37B51, 68Q80}
\keywords{decidability, non-uniform cellular automata}

\begin{abstract}
For non-uniform cellular automata (NUCA) with finite memory over an arbitrary universe with multiple local transition rules, we show that pointwise nilpotency, pointwise periodicity, and pointwise eventual periodicity properties are respectively equivalent to nilpotency, periodicity, and eventual periodicity. Moreover, we prove that every linear NUCA which satisfies pointwise a polynomial equation (which may depend on the configuration) must be an eventually periodic linear NUCA. Generalizing results for higher dimensional group and linear CA, we also establish the decidability results of the above dynamical properties as well as the injectivity for arbitrary  NUCA with finite memory which are local perturbations of higher dimensional linear and group CA. Some generalizations to the case of sparse global perturbations of higher dimensional linear and group CA are also obtained. 
\end{abstract}

\date{\today}
\maketitle
  
\setcounter{tocdepth}{1}

\section{Introduction}
\par 
Since the original works of von Neumann and Ulam \cite{neumann}, cellular automata (CA) arises as a powerful model of various physical, biological, and social phenomena   where homogeneous cells evolve according to the same local transition rule. As a model of distributed computing, CA are known to be Turing complete \cite{GOL}. Whenever the  cells happen to follow, either intentionally or as a result of some undesirable perturbation, different local transition rules, we obtain the more general class of non-uniform CA (NUCA) where  the uniformity between the cells is broken. It is therefore important to understand the behaviour of the class of NUCA especially those which are local perturbations of CA. For instance, it was shown in \cite{phung-shadowing} and \cite{phung-dual-nuca} that not only higher dimensional linear CA but the class of higher dimensional linear NUCA also satisfy the remarkable pseudo-orbit tracing property. While almost all  properties of CA are undecidable \cite{kari-nilpotent}, \cite{kari-reversible}, results in \cite{formenti-hoca}, \cite{formenti-ergodicity}, \cite{kari-beaur} show that for linear and more generally group CA, many dynamical properties turn out to be decidable. In this article, we will show that several  dynamical properties are still decidable for local perturbations, and more generally, sparse global perturbations of higher dimensional  linear CA (Theorem~\ref{t:intro-decidable-local-perturbation-lnuca}, Theorem~\ref{t:intro-decidable-sparse-perturbation-lnuca}). We also establish a general result saying that many pointwise properties of NUCA are equivalent to their uniform counterparts (Theorem~\ref{t:intro-pointwise-cayley-hamilton-linear-nuca}, Theorem~\ref{t:intro-pointwise-uniform-nuca}). 
\par 
To state the main results, recall that a CA over a group $G$ (the \emph{universe}) and a set $A$ (the \emph{alphabet}) as  a self-map $A^G \righttoleftarrow$ which is $G$-equivariant and uniformly continuous (cf.~\cite{hedlund-csc}, \cite{hedlund}). 
Here, we equip the \emph{full shift} $A^G$, i.e., the set of all maps $c \colon G \to A$ called \emph{configurations},  with the prodiscrete topology. The \emph{Bernoulli shift} action $G \times A^G \to A^G$ is defined by $(g,x) \mapsto g x$, 
where $(gx)(h) =  x(g^{-1}h)$ for  $g,h \in G$ and $x \in A^G$. 
When different cells are allowed to admit different local transition maps, we have the more general notion of NUCA (cf. \cite[Definition~1.1]{phung-tcs}, \cite{Den-12a}, \cite{Den-12b}).   

\begin{definition}
\label{d:most-general-def-asyn-ca}
Let $G$ be a group and let  $A$ be a set. Let $M \subset G$ be a subset and let $S = A^{A^M}$ be the set of all maps $A^M \to A$. Given $s \in S^G$, the NUCA $\sigma_s \colon A^G \to A^G$ is defined for all $x \in A^G$ and $g \in G$ by the formula 
\begin{equation*}
\sigma_s(x)(g)=  
    s(g)((g^{-1}x)  
	\vert_M). 
 \end{equation*}
 \end{definition} 
\par 
The set $M$ is called a \emph{memory} and $s \in S^G$  the \textit{configuration of local defining maps} of $\sigma_s$. It follows that  every CA is a NUCA with finite memory and constant configuration of local defining maps. 
\par 
For $x \in A^G$, let $\Sigma(x) = \overline{\{gx \colon g \in G\}} \subset A^G$ be the smallest subshift which contains $x$. Two configurations $x,y  \in A^G$ are \emph{asymptotic} if  $x\vert_{G \setminus E}=y\vert_{G \setminus E}$ for some finite subset $E \subset G$.  
Following \cite{phung-tcs}, we say that $\sigma_s$ is \emph{post-surjective} if for all $x, y \in A^G$ with $y$ asymptotic to $\sigma_s(x)$, there exists $z \in A^G$ asymptotic to $x$ such that  $ \sigma_s(z)=y$. 
Recall also the following class $\mathrm{LNUCA}_c$ (cf.~\cite[Definition~1.2]{phung-twisted-group-ring}) consisting of  local perturbations of linear CA. 
\begin{definition}
Let  $G$ be a group $G$ and let $V$ be a vector space $V$. We denote by $\mathrm{LNUCA}_{c}(G, V)$ the space of all linear NUCA $\tau \colon V^G \to V^G$ with finite memory which admit asymptotically constant configurations of local defining maps, i.e., $\tau \in \mathrm{LNUCA}_{c}(G, V)$ if there exist  finite subsets $M, E \subset G$ and $s \in \LL(V^M, V)^G$ such that $\tau=\sigma_s$ and $s(g)=s(h)$ for all $g,h \in G \setminus E$.  
\end{definition}
\par 
Our first main result is the following decidability theorem (see also Theorem~\ref{t:general-decidable-local-perturbation-lnuca} for a more general statement) which generalizes the corresponding result for linear CA in \cite{kari-beaur}.  
\begin{theoremletter}
\label{t:intro-decidable-local-perturbation-lnuca} 
Let $G= \Z^d$, $d \geq 1$, and let $V$ be a finite vector space. Let $s \in \mathrm{LNUCA}_c(G,V)$. Then it is decidable whether: 
 \begin{enumerate} [\rm (i)]
     \item  $\sigma_s$ is nilpotent; 
     \item  $\sigma_s$ is (eventually) periodic;  
     \item  $\sigma_s$ is Cayley-Hamilton. 
     \item 
     $\sigma_s$ is injective; 
     \item 
     $\sigma_s$ is post-surjective. 
 \end{enumerate}
\end{theoremletter}
\par 
Here, a linear NUCA $\tau$ is said to be nilpotent, resp. periodic, resp. eventually periodic, resp. Cayley-Hamilton if there exist  integers $m, n \geq 1$ and a nonzero polynomial $P$ such that $\tau^n=0$, resp. $\tau^n=1$, resp. $\tau^{m+n}=\tau^n$, resp. $P(\tau)=0$.   
In fact, we can easily modify  the proof of Theorem~\ref{t:intro-decidable-local-perturbation-lnuca} to obtain a more general result, except for the injectivity property, concerning the wider class of sparse global perturbations of linear CA 
(see Definition~\ref{d:sparse-perturbations}). More precisely, a linear NUCA $\sigma_s$ is a sparse global perturbation of a linear CA $\tau = \sigma_c$ if there exists a finite subset $E\subset G$ and a sequences $(g_n)_{n \geq 0}$ of elements in $G$ such that  $s\vert_{G \setminus F}= c\vert_{G \setminus F}$ where $F = \cup_{n \geq 0} g_n E$ and 
for every finite subset $D \subset G$, there exists $n_0\geq 0$ such that $g_mD  \cap g_n D = \varnothing$  for all $m, n \geq n_0$. Then our obtained result is as follows  (see Theorem~\ref{t:decidable-sparse-perturbation-lnuca-nilpotent}): 
\par 
\begin{theoremletter}
\label{t:intro-decidable-sparse-perturbation-lnuca} 
Let $M$ be a finite subset of $G= \Z^d$, $d \geq 1$, and let $V$ be a finite vector space. Let $s \in \LL(V^M,V)^G$ and suppose that $\sigma_s$ is a sparse global perturbation of a linear CA.  Then it is decidable whether: 
 \begin{enumerate} [\rm (i)]
     \item  $\sigma_s$ is nilpotent; 
     \item  $\sigma_s$ is (eventually)  periodic;  
     \item  $\sigma_s$ is Cayley-Hamilton. 
 \end{enumerate}
\end{theoremletter}
\par 
We say that a linear NUCA $\tau$ is pointwise nilpotent if for every configuration $x$, we have $\tau^n(x)=0$ for some integer $n \geq 1$. Other pointwise properties such as pointwise Cayley-Hamilton and pointwise (eventual) periodic properties are defined analogously. Theorem~\ref{t:finite-cayley-hamilton-linear-ca-periodic} shows that over a finite vector space alphabet, all pointwise Cayley-Hamilton linear NUCA are eventually periodic. This result clearly fails over infinite vector space alphabet. 
\par 
For general CA, it is well-known that many pointwise properties are uniform.  
We  establish a the following result for linear NUCA over an arbitrary (not necessarily finite) finite-dimensional vector space alphabet: 
\begin{theoremletter}
    \label{t:intro-pointwise-cayley-hamilton-linear-nuca} 
Let $\tau$ be a linear NUCA with finite memory  over a countable group universe and a countable finite dimensional vector space alphabet.   
Suppose that $\tau$ is pointwise Cayley-Hamilton, resp.  pointwise (eventually) periodic, resp. pointwise nilpotent. Then $\tau$ is   Cayley-Hamilton, resp. (eventually) periodic, resp. nilpotent. 
\end{theoremletter}
\par 
When the alphabet is finite, we have the following similar result for the class of all NUCA with finite memory over a countable group universe.   
\begin{theoremletter}
\label{t:intro-pointwise-uniform-nuca} 
Let $M \subset G$ be a finite subset of a countable group $G$ and let $A$ be a finite alphabet. Let $S=A^{A^M}$ and let $s\in S^G$.  Suppose that $\sigma_s$ is pointwise nilpotent, resp.  pointwise periodic., resp. pointwise eventually periodic.  Then $\sigma_s$ is  nilpotent, resp.  periodic., resp.  eventually periodic. 
\end{theoremletter}

\par 
The paper is organized as follows. In Section~\ref{s:induced-local-map}, we recall briefly the useful construction of various induced local maps of NUCA. In Section~\ref{s:3}, we study nilpotent and pointwise nilpotent NUCA and show that the two properties are stable and in fact equivalent (Lemma~\ref{l:stable-nilpotency-nuca}, Lemma~\ref{l:stable-periodic-nuca}, Theorem~\ref{t:pointwise-uniform-nilpotency-nuca}). Similarly, we establish the same equivalence for (eventual) periodicity and pointwise (eventual) periodicity (Theorem~\ref{t:pointwise-periodic-nuca}, Theorem~\ref{t:pointwise-uniform-eventually-periodic-nuca}) in Section~\ref{s:4} and  Section~\ref{s:5}. Then in Section~\ref{s:6}, we show that over finite vector space alphabet, every pointwise Cayley-Hamilton linear NUCA must be eventually periodic (Theorem~\ref{t:finite-cayley-hamilton-linear-ca-periodic}). For arbitrary finite-dimensional vector space alphabe, we then establish the equivalence between pointwise nilpotency, pointwise (eventual) periodicity, pointwise Cayley-Hamilton, and their respective uniform counterparts (Theorem~\ref{t:pointwise-cayley-hamilton-linear-nuca}, Theorem~\ref{t:pointwise-uniform-eventually-periodic-linear-nuca}, Theorem~\ref{t:pointwise-uniform-periodic-linear-nuca}, Theorem~\ref{t:pointwise-uniform-nilpotency-linear-nuca}). These results are generalized in  
Theorem~\ref{t:general-cayley-hamilton-sequence}. For the class $\mathrm{LNUCA}_c$, we show the nilpotency and the (eventual) periodicity are decidable properties (Theorem~\ref{t:decidable-lnuca-c-periodic}, Theorem~\ref{t:decidable-group-nuca-c-nilpotent}). 
The proof of Theorem~\ref{t:intro-decidable-sparse-perturbation-lnuca} is given in Section~\ref{s:8}.  
Finally, we prove that injectivity is also a decidable property for elements in the class $\mathrm{LNUCA}_c$ over finite vector space alphabet is obtained in Theorem~\ref{t:main-injectivity} by generalizing the effective constructions in \cite{kari-beaur} to the case of monoid universe  to notably cover the projections to the half spaces.

\section{Induced local maps of NUCA} 
\label{s:induced-local-map}

Let $G$ be a group and let $A$ be a set. For every subset $E\subset G$, $g \in G$,  and $x \in A^E$ we define $gx \in A^{gE}$ by setting $gx(gh)=x(h)$ for all $h \in E$. In particular,  $gA^E= \{gx \colon x \in A^E\}=A^{gE}$. 
\par 
Let $M$  be a subset of a group $G$. Let $A$ be a set and let $S=A^{A^M}$. For every finite subset $E \subset G$  and $w \in S^{E}$,  
we define a map  $f_{E,w}^+ \colon A^{E M} \to A^{E}$ for every $x \in A^{EM}$ and $g \in E$ by (in the case of CA, see e.g. \cite[Lemma~3.2]{cscp-alg-goe}, \cite[Proposition~3.5]{phung-2020}, \cite[Secion~2.2]{phung-embedding}):  
\begin{align}
\label{e:induced-local-maps} 
    f_{E,w}^+(x)(g) & = w(g)((g^{-1}x)\vert_M). 
\end{align}
\par 
In the above formula, note that  $g^{-1}x \in A^{g^{-1}EM}$ and $M \subset g^{-1}EM$ since $1_G \in g^{-1}E$ for $g \in E$. Therefore, the map   $f_{E,w}^+ \colon A^{E M} \to A^{E}$ is well defined. 
\par
Consequently, for every $s \in S^G$, we have a well-defined induced local map $f_{E, s\vert_E}^+ \colon A^{E M} \to A^{E}$ for every finite subset $E \subset G$ which satisfies: 
\begin{equation}
\label{e:induced-local-maps-general} 
    \sigma_s(x)(g) =  f_{E, s\vert_E}^+(x\vert_{EM})(g)
\end{equation}
for every $x \in A^G$ and $g \in E$. Equivalently, we have for all $x \in A^G$ that: 
\begin{equation}
\label{e:induced-local-maps-proof} 
    \sigma_s(x)\vert_E =  f_{E, s\vert_E}^+(x\vert_{EM}). 
\end{equation}
\par 
It follows that for $\Gamma = \im \sigma_s$, we have the relation: 
\begin{equation} 
\label{e:local-image}
    \Gamma_E = \im f_{E, s\vert_E}^+. 
\end{equation} 

\section{Nilpotent and pointwise nilpotent NUCA} 
\label{s:3}

We establish the following useful  result which extends \cite[Lemma~4.1]{phung-tcs}. 

\begin{lemma}
    \label{l:power-almost-g-invariant} 
    Let $M \subset G$ be a  subset of a group $G$ and let $A$ be an alphabet (not necessarily finite). Let $S= A^{A^M}$ and let $s \in S^G$. Then for every site $g \in G$,  every configuration $x \in A^G$, and every integer $n \geq 0$, we have: 
    \begin{equation}
    \label{e:l:power-almost-g-invariant}
        \sigma_{gs}^n(gx) = g\sigma_s^n(x). 
    \end{equation}
\end{lemma}

\begin{proof}
We will prove by induction that $\sigma_{gs}^n(gx) = g \sigma_s^n(x)$ for every integer $n \geq 1$ and all $g \in G$,  $x \in A^G$. The case $n=0$ is trivial and the case $n=1$ results directly from \cite[Lemma~4.1]{phung-tcs}. Suppose that the formula \eqref{e:l:power-almost-g-invariant} holds for some integer $k \geq 1$. Then we have $\sigma_{gs}^k(gy)= g\sigma_s^k(y)$ for all $y \in A^G$ and $g \in G$. In particular, for $y= \sigma_s(x)$, we obtain $\sigma_{gs}^k(g\sigma_s(x))= g\sigma_s^k(\sigma_s(x))$. Combining with the equality $\sigma_{gs}(gx) = g \sigma_s(x)$, we can thus compute: 
    \begin{align*}
      \sigma_{gs}^{k+1}(gx) & = \sigma_{gs}^k (\sigma_{gs}(gx)) \\
      & = \sigma_{gs}^k (g\sigma_{s}(x)) \\
      & = g \sigma_s^k (\sigma_s(x))\\
      & = g \sigma_s^{k+1}(x).
    \end{align*}
\par 
The proof is thus complete. 
\end{proof}

\begin{definition}
\label{d:pointwise-nil-nuca} 
Let $G$ be a group and let $A$ be an alphabet. Let $\sigma_s \colon A^G \to A^G$ be an NUCA. We say that $\sigma_s$ is \emph{pointwise nilpotent} if there exist a constant configuration $x_0 \in A^G$ such that for every $x\in A^G$, there exists an integer $n \geq 1$ depending on $x$ such that $\sigma_s^n(x)=x_0$. We say that $\sigma_s$ is \emph{nilpotent} if there exist a constant configuration $x_0 \in A^G$ and an integer $n \geq 1$ such that $\sigma_s^n(x)=x_0$ for every $x \in A^G$.
\end{definition}
\par 
If $\sigma_s$ is a linear NUCA, then it is clear that the constant configuration $x_0$ in the above definition must be $0^G$. We remark the following elementary property. 
\begin{lemma}
    \label{l:nilpotency-tau-tau-k}
    Let $k \geq 1$ be an integer. Then an NUCA $\tau$ is pointwise nilpotent, resp. nilpotent, if and only if so is $\tau^k$.
\end{lemma}

\begin{proof}
Let $\sigma_s$ be an NUCA. It is trivial that if $\sigma_s^k$ is pointwise nilpotent, resp. nilpotent, then so is $\sigma_s$. Conversely, suppose that $\sigma_s$ is pointwise nilpotent, resp. nilpotent, and let $x_0$ be the constant configuration as in Definition~\ref{d:pointwise-nil-nuca}. Then by definition, we deduce that  $\sigma_s^{n}(x_0)=x_0$ for a suitable $n_0 \geq 1$. It is then clear that $\sigma_s^k$ is also pointwise nilpotent, resp. nilpotent, with the same $x_0$ as the constant configuration satisfying Definition~\ref{d:pointwise-nil-nuca}. 
\end{proof}

\par 
 The following result implies that nilpotency is a stable property. 
\begin{lemma}
    \label{l:stable-nilpotency-nuca} 
Let $M \subset G$ be a finite subset of a countable group $G$ and let $A$ be an alphabet (not necessarily finite). Let $S= A^{A^M}$ be the set of all maps $A^M \to A$. Let $s \in S^G$ and suppose that $\sigma_s$ is nilpotent, i.e., there exist 
a constant configuration $x_0 \in A^G$ and an integer $n \geq 1$ such that $\sigma_s^n(x)=x_0$ for all $x \in A^G$. Then for every $ t \in \Sigma(s)$, we have $\sigma_t^n(x)=x_0$ for all $x \in A^G$.
\end{lemma}

\begin{proof}
For every $g \in G$ and every $x \in A^G$, we infer from Lemma~\ref{l:power-almost-g-invariant} that:
\begin{equation}
\label{l:stable-nilpotency-nuca-1-2}
\sigma_{gs}^n(x) = g \sigma_s^n(g^{-1}x)=gx_0=x_0.  
\end{equation}
\par 
The last equality results from the hypothesis that $x_0$ is a constant configuration. 
Hence, $\sigma_{gs}^n(x)=x_0$ for all $g \in G$ and $x \in A^G$.    
Let $t \in \Sigma(s)$ and let $h \in G$. By \cite[Theorem~6.2]{phung-tcs}, we know that $N=M^n$ is a finite memory set of $\sigma_t^n$ and $\sigma_{gs}^n$ for all $g \in G$. 
Since $t \in \Sigma(s)$, we can find $g \in G$ such that  $t\vert_{hN}=(gs)\vert_{hN}$. 
We can write  $\sigma_t^n= \sigma_v$ and $\sigma_{gs}^n=\sigma_u$ for some  $u, v \in U^G$ where $U=A^{A^N}$. 
It follows from the equality $t\vert_{hN}=(gs)\vert_{hN}$ that $u(h)= v(h)$. As $\sigma_v=\sigma_{gs}^n$ is the constant map with value the constant configuration $x_0$ by \eqref{l:stable-nilpotency-nuca-1-2}, we deduce that $u(h) \colon {A^N} \to A$ is the constant map with value $x_0(1_G)$. Therefore, $v(h) \colon {A^N} \to A$ is also constant with value $x_0(1_G)$ for every $h \in G$. We can thus conclude that 
$\sigma_t^n(x)=\sigma_v(x)=x_0$ for all $x \in A^G$ and the proof is complete. 
\end{proof}

\par 
Similarly, we have the following result: 
\begin{lemma}
    \label{l:stable-periodic-nuca} 
Let $M $ be a finite subset of a  group $G$ and let $A$ be an alphabet (not necessarily finite). Let $S= A^{A^M}$ be the set of all maps $A^M \to A$. Let $s \in S^G$ and suppose that  there exists  
 an integers $m,n \geq 0$ such that $\sigma_s^m = \sigma_s^n$. Then for every $ t \in \Sigma(s)$, we have $\sigma_t^m=\sigma_t^n$.  
\end{lemma}

\begin{proof}
Since $\sigma_s^m=\sigma_s^n$, we deduce from  Lemma~\ref{l:power-almost-g-invariant} that $\sigma_{gs}^m=\sigma_{gs}^n$ for all $g \in G$ since for every configuration $x \in A^G$, we have: 
\begin{equation*}
\sigma_{gs}^m(x) = g \sigma_s^m(g^{-1}x)=g \sigma_s^n(g^{-1}x)= \sigma_{gs}^n(x).  
\end{equation*}
\par 
Let $t \in \Sigma(s)$ and let $h \in G$. Then for  $N=M^{\max(m,n)}$, we can find as in the proof of Lemma~\ref{l:stable-nilpotency-nuca} some $g \in G$ and $u, v, w \in U^G$, where $U=A^{A^N}$, such that  $t\vert_{hN}=(gs)\vert_{hN}$ and 
\[
\sigma_{gs}^m=\sigma_{gs}^n=\sigma_u, \quad   
\sigma_t^m=\sigma_v, \quad \text{and} \quad  \sigma_t^n=\sigma_w.
\] 
\par 
Consequently, $u(h)= v(h)=w(h)$ for every $h \in G$ and we can thus conclude that 
$\sigma_t^m=\sigma_t^n$. The proof is complete. 
\end{proof}

\par 
The following theorem states that every pointwise nilpotent NUCA over an arbitrary finite alphabet must be nilpotent.  

\begin{theorem}
\label{t:pointwise-uniform-nilpotency-nuca} 
Let $M \subset G$ be a finite subset of a countable group $G$ and let $A$ be a finite alphabet. Let $S=A^{A^M}$ and let $s\in S^G$.   Suppose that the NUCA $\sigma_s$ is pointwise nilpotent. Then $\sigma_s$ is nilpotent. Moreover, if $\sigma_s^n$ is constant whose image is  a constant configuration then so is $\sigma_t^n$ for every $t \in \Sigma(s)$. 
\end{theorem}

\begin{proof}
Up to enlarging $M$, we can suppose without loss of generality that $1_G \in M$ and $M=M^{-1}$. 
Since $\sigma_s$ is pointwise nilpotent, there exists a constant configuration $x_0 \in A^G$ such that for every $x \in A^G$, we can find some integer $n \geq 1$ such that $\sigma_s^n(x)=x_0$. Let $r \geq 1$ be an integer such that 
$\sigma_s^r(x_0)=x_0$.  
\par 
By the continuity of $\sigma_s$ (cf.~\cite[Lemma~4.1]{phung-tcs} and \cite[Theorem~6.2]{phung-tcs}), we obtain for every $k \geq 1$ a closed subset $X_k$  of $A^G$ where: 
    \[
X_k \coloneqq (\sigma_s^{kr})^{-1}(x_0) = \{ x \in A^G \colon \sigma_s^{kr}(x) = x_0\}. 
    \]
\par 
Since $\sigma_s$ is pointwise nilpotent by hypothesis, $A^G = \cup_{k=1}^{\infty} X_k$. 
Let $x \in X_k$ for some $k \geq 1$ then we have $\sigma_s^{kr}(x)=x_0$. We find that $x \in X_{k+1}$ since
\begin{align*}
\sigma_s^{(k+1)r}(x) = \sigma_s^r(\sigma_s^{kr}(x))=\sigma_s^r(x_0) = x_0. 
\end{align*}
\par 
Consequently, $(X_k)_{k \geq 1}$ forms an exhaustion by closed subsets of $A^G$. Note that $A^G$ admits a complete metric compatible with its prodiscrete topology as it is a countable product of discrete spaces. In particular, $A^G$ is a Baire space and we infer from the Baire category theorem that the interior of $X_{n_0}$ is nonempty for some $n_0 \geq 1$ and thus there exist a  finite subset $E \subset G$ and a pattern $p \in V^E$ such that  $W= \{p\} \times A^{G \setminus E} \subset X_{n_0}$. 
\par 
By \cite[Theorem~6.2]{phung-tcs}, there exists $t \in T^G$, where $T=A^{A^N}$ with $N=M^{rn_0}$, such that $\sigma_t= \sigma_s^{rn_0}$.  Note that $N$ is a finite symmetric subset of $G$, i.e., $N=N^{-1}$, and $1_G \in N$. Moreover, since $\sigma_s^r(x_0)=x_0$, we deduce that $x_0$ is a fixed point of $\sigma_t$, i.e., $\sigma_t(x_0)=x_0$.  
It follows from the above paragraph that for every $x \in W$, we have  $\sigma_t(x)= x_0$. 
Therefore, for all $g \in G \setminus EN$ and for all configuration $x \in W$, we have: 
\begin{align}
\label{e:pointwise-uniform-nilpotency-nuca-general}
    t(g)((g^{-1}x)\vert_N) = \sigma_t(x)(g)= x_0(g) = x_0(1_G). 
\end{align}
\par 
Since $g \in G \setminus EN$, we have $gN \cap E = \varnothing$. Thus, $(g^{-1}W)\vert_M=A^M$ and we deduce from \eqref{e:pointwise-uniform-nilpotency-nuca-general} that $t(g)(z)=x_0(1_G)$ for all $z \in A^M$ and $g \in G \setminus EN$. In other words, $t(g)$ is the constant function with value $x_0(1_G)$ for every $g \in G \setminus EN$. Therefore, $\sigma_t(x)\vert_{G \setminus EN} = x_0\vert_{G \setminus EN}$ for every $x \in A^G$.  
\par 
Since $\sigma_s$ is pointwise nilpotent, so is $\sigma_t$ and there exist integers $n_w \geq 1$ where $w \in A^{EN}$  such that 
$\sigma_t^{n_w}(\tilde{w})= x_0$ where $\tilde{w}=w \times x_0\vert_{G \setminus EN}$. Since $\sigma_t(x_0)=x_0$,  it follows that $\sigma_t^{n}(\tilde{w})=x_0$  for all $n \geq n_w$. 
\par 
Consequently, if we denote 
$q = 1+  \max \{n_w \colon w \in A^{EN} \}$ then $q \in \N$ since $A$ and $EN$ are finite, and we claim that $\sigma_t^q(x)=x_0$ for all $x \in A^G$. Indeed, let $x \in A^G$ and $w = \sigma_t(x)\vert_{G \setminus EN} \in A^{EN}$. We have shown above that $\sigma_t(x)\vert_{G \setminus EN}= x_0\vert_{G \setminus EN}$. Hence, $\sigma_t(x)= \tilde{w}$.  As $q-1 \geq n_w$, we find that: 
\begin{align*}
    \sigma_{t}^q(x) = \sigma_t^{q-1}(\sigma_t(x)) = \sigma_t^{q-1}(\tilde{w})= x_0. 
\end{align*}
\par 
We conclude that $\sigma_s^{qrn_0}(x)= \sigma_t^q(x)=x_0$ for all $x \in A^G$. Hence, $\sigma_s$ is nilpotent. Since the last statement follows immediately from Lemma~\ref{l:stable-nilpotency-nuca},  the proof is complete. 
\end{proof}

\begin{corollary}
    An NUCA with finite memory over a finite alphabet and a countable group universe is nilpotent if and only if it is pointwise nilpotent.  
\end{corollary}

\begin{proof}
Suppose that $\tau$ is an NUCA with finite memory over a finite alphabet and a countable group universe. Then Theorem~\ref{t:pointwise-uniform-nilpotency-nuca} shows that $\tau$ is nilpotent whenever it is pointwise nilpotent. Since the converse is trivial, the proof of the corollary is complete.  
\end{proof}

\section{Periodic and pointwise periodic NUCA} 
\label{s:4}
As for pointwise nilpotency, we have the following natural definition of pointwise periodic NUCA. 

\begin{definition}
\label{d:pointwise-per-nuca} 
Let $G$ be a group and let $A$ be an alphabet. Let $\sigma_s \colon A^G \to A^G$ be an NUCA. We say that $\sigma_s$ is \emph{pointwise periodic} if  for every $x\in A^G$, there exists an integer $n \geq 1$ depending on $x$ such that $\sigma_s^n(x)=x$. We say that $\sigma_s$ is \emph{periodic} if there exists an integer $n \geq 1$ such that $\sigma_s^n(x)=x$ for every $x \in A^G$. 
\end{definition}

\par 
We begin with a simple but useful  observation.
\begin{lemma}
\label{l:simple-tau-tau-k-per}
Let $k \geq 1$ be an integer. Then  an NUCA $\tau$ is periodic, resp. pointwise periodic, if and only if so is $\tau^k$. 
\end{lemma}

\begin{proof}
It suffices to note that if $\tau^m(x)=x$ for some configuration $x$ and some integer $m \geq 1$ then we have by an immediate induction:  
\[
(\tau^k)^m(x) =(\tau^{km})(x) = \tau^{(k-1)m}(\tau^m(x))= \tau^{(k-1)m}(x)=\dots=\tau^m(x)=x. 
\] 
\end{proof}

In parallel with Theorem~\ref{t:pointwise-uniform-nilpotency-nuca}, we establish the following equivalence between periodicity and  pointwise periodicity for any NUCA with finite memory and finite alphabet.    

\begin{theorem}
\label{t:pointwise-periodic-nuca} 
Let $M \subset G$ be a finite subset of a countable group $G$ and let $A$ be a finite alphabet. Let $S=A^{A^M}$ and let $s\in S^G$.   Suppose that the NUCA $\sigma_s$ is pointwise periodic. Then $\sigma_s$ is periodic. 
\end{theorem}

\begin{proof}
Up to enlarging $M$, we can suppose without loss of generality that $1_G \in M$ and $M=M^{-1}$. 
By the continuity of $\sigma_s$ (cf.~\cite[Lemma~4.1]{phung-tcs}), we obtain for every $k \geq 1$ a closed subset $X_k$  of $A^G$ where: 
    \[
X_k \coloneqq  \{ x \in A^G \colon \sigma_s^{k}(x) = x\}. 
    \]
\par 
As in the proof of Theorem~\ref{t:pointwise-uniform-nilpotency-nuca}, we find that $X_{k} \subset X_{kl}$ for all integers $k,l \geq 1$ and that we infer from the Baire category theorem there exist an integer  $n_0 \geq 1$, a finite subset $E \subset G$, and a pattern $p \in V^E$ such that  $W= \{p\} \times A^{G \setminus E} \subset X_{n_0}$.  
By \cite[Theorem~6.2]{phung-tcs}, there exists $t \in T^G$, where $T=A^{A^N}$ with $N=M^{n_0}$, such that $\sigma_t= \sigma_s^{n_0}$. 
It follows that $\sigma_t(x)= x$ for every $x \in W$. 
Hence, for $g \in G \setminus EN$ and $x \in W$, we find that: 
\begin{align}
\label{e:pointwise-uniform-periodic-nuca-general} 
t(g)((g^{-1}x)\vert_N) = \sigma_t(x)(g)= x(g). 
\end{align}
\par 
Note that $gN \cap E = \varnothing$ as $g \in G \setminus EN$. Thus, we have $(g^{-1}W)\vert_M=A^M$ and  \eqref{e:pointwise-uniform-periodic-nuca-general} then implies that $t(g) \colon A^{N} \to A$ is the canonical projection $A^N \to A^{\{1_G\}}$ induced by the inclusion 
$\{1_G\} \subset N$. Consequently, we have $\sigma_t(x)\vert_{G \setminus EN} = x\vert_{G \setminus EN}$ for all $x \in A^G$. 
Let us fix a state $a \in A$. 
Since $\sigma_s$ is pointwise periodic, so is $\sigma_t$ and there exist integers $n_w \geq 1$ where $w \in A^{EN^2}$  such that 
$\sigma_t^{n_w}(\tilde{w})= \tilde{w}$  where $\tilde{w}=w \times a^{G \setminus EN^2}$.  Since  
 $\sigma_t(y)\vert_{G \setminus EN}= y\vert_{G \setminus EN}$ for every configuration $y \in A^G$, we deduce that $\sigma_t^{k}(z)\vert_{EN^2 \setminus EN} = w\vert_{EN^2\setminus EN}$ for every $k \geq 1$ and $z \in A^G$ such that $z\vert_{EN^2\setminus EN} = w\vert_{EN^2 \setminus EN}$.  Consequently, as $N$ is a memory of $\sigma_t$ and $EN \subset EN^2$, we find that 
 \begin{align}
 \label{e:t:pointwise-periodic-nuca-1-2-3}
 \sigma_t^{kn_w}(z)\vert_{EN}=  \sigma_t^{kn_w}(\tilde{w})\vert_{EN}=w\vert_{EN}
 \end{align}
 for every $k \geq 1$ and every  configuration $z \in A^G$ such that $z\vert_{EN^2}= w\vert_{EN^2}$. 
\par 
Since $A$ and $EN$ are finite, we obtain a finite integer
\[ 
q = \prod_{w \in A^{EN^2}} n_w  \geq 1.
\]
\par 
We shall prove that $\sigma_t^q= \Id$. Indeed, let $x \in A^G$ and $w = x\vert_{EN^2} \in A^{EN^2}$. For every $k \geq 1$, we have seen  that 
\begin{equation}
    \label{e:t:pointwise-periodic-nuca-1-2-3-4-5}\sigma_t^k(x)\vert_{G \setminus EN}= x\vert_{G \setminus EN}. 
\end{equation}
\par 
On the other hand, for $k_w = \frac{q}{n_w} \in \N$, we infer from \eqref{e:t:pointwise-periodic-nuca-1-2-3} that: 
\begin{align}
\label{e:t:pointwise-periodic-nuca-1-2-3-4}
    \sigma_{t}^q(x)\vert_{EN} = \sigma_t^{k_w n_w}(x)\vert_{EN} =  \sigma_t^{k_w n_w}(\tilde{w})\vert_{EN} = w\vert_{EN}= x\vert_{EN}. 
\end{align}
\par 
We can thus conclude from \eqref{e:t:pointwise-periodic-nuca-1-2-3-4-5} and \eqref{e:t:pointwise-periodic-nuca-1-2-3-4} that $\sigma_s^{qn_0}(x)= \sigma_t^q(x)=x$ for all $x \in A^G$. Hence, $\sigma_s^{qn_0}=\sigma_t^q=\Id$  and the proof is complete. 
\end{proof}

\section{Eventually periodic NUCA}
\label{s:5}
To generalize the periodicity properties of NUCA, we have  the  notion of eventual periodic NUCA which holds only after some iteration of the NUCA.    

\begin{definition}
\label{d:pointwise-eventually-polynomial-lnuca-2} Let $\sigma_s \colon A^G \to A^G$ be an NUCA over a group universe $G$ and an alphabet $A$.  We say that $\sigma_s$ is \emph{pointwise eventually periodic} if  for every $x\in A^G$, there exist  integers $n,m \geq 1$ depending on $x$ such that $\sigma_s^{n+m}(x)=\sigma_s^m(x)$. We say that $\sigma_s$ is \emph{eventually periodic} if there exist integers $m,n \geq 1$ such that $\sigma_s^{n+m}(x)=\sigma_s^m(x)$ for all $x \in A^G$. 
\end{definition}

\par 
As for periodicity and pointwise periodicity, we also have the following simple observation similar to Lemma~\ref{l:simple-tau-tau-k-per}. 
\begin{lemma}
\label{l:simple-periodic-observation}
Let $k \geq 1$ be an integer. Then an NUCA $\tau$ is eventually periodic, resp. pointwise eventually periodic, if and only if so is $\tau^k$. Moreover, $\tau$ is eventually periodic if and only if $\tau^{2n}=\tau^n$ for some integer $n \geq 1$. 
\qed 
\end{lemma}

\par 
The main result of the section is the following which states that pointwise eventual periodicity is equivalent to eventual periodicity. The proof is somewhat more complicated than the proof of Theorem~\ref{t:pointwise-periodic-nuca}. 

\begin{theorem}
\label{t:pointwise-uniform-eventually-periodic-nuca} 
Let $M \subset G$ be a finite subset of a countable group $G$ and let $A$ be a finite alphabet. Let $S=A^{A^M}$, $s \in S^G$, and suppose that $\sigma_s$ is pointwise eventually periodic. Then $\sigma_s$ is eventually  periodic. 
\end{theorem}

\begin{proof}
Up to enlarging $M$, we can suppose that $1_G \in M$. For integers $k, l \geq 1$, we define:
    \begin{align*}
X_{k,l}   = \{ x \in V^G \colon \sigma_s^{k+l}(x) = \sigma_s^k(x)\}. 
    \end{align*}
\par 
Then $X_{k,l} \subset V^G$ is closed by the continuity of $\sigma_s^q$ for all $q \geq 1$ (cf.~\cite[Lemma~4.1]{phung-tcs}). 
Since $\sigma_s$ is pointwise eventually periodic by hypothesis, we have $V^G = \cup_{k, l \geq 1}^{\infty} X_{k,l}$. Hence, the Baire category theorem implies  that there exist $k_0, l_0 \geq 1$ such that the interior of $X_{k_0,l_0}$ is nonempty. Consequently, we can find a finite subset $E \subset G$ and a pattern $p \in V^E$ such that $W= \{p\} \times V^{G \setminus E} \subset X_{k_0, l_0}$. Let $N=M^{\max(k_0,l_0)}$ then $1_G \in N$  and we infer from  \cite[Theorem~6.2]{phung-tcs} that there exists $t \in T^G$, where $T=A^{A^N}$, such that 
\[
\sigma_t= \sigma_s^{l_0}.
\] 
\par 
For every $x \in W$, we have $\sigma_s^{k_0+l_0}(x)= \sigma_s^{k_0}(x)$ and thus $\sigma_s^{ql_0}(\sigma_s^{k_0}(x))= \sigma_s^{k_0}(x)$ for all $q \geq 0$. 
For every $k \geq 1$, we denote $\Gamma(k) = \sigma_s^{k}(A^G) \subset A^G$ and $\Lambda(k)=\sigma_s^{k}(W) \subset A^G$. Then it follows that for every integer $q \geq 0$, we have: 
\begin{align}
\label{e:pointwise-uniform-eventually-periodic-nuca=1-a}
\sigma_t^{q}\vert_{\Lambda(k_0)} = \Id_{\Lambda(k_0)}.  
\end{align} 
\par 
We note that $\Gamma(k)\subset \Gamma(l)$ for all $k \geq l \geq 1$. Moreover, by definition of $W$ and by the choice of $N$, we have for every element $g \in G \setminus EN^2$ that: 
\begin{equation}
\label{e:pointwise-uniform-eventually-periodic-nuca=1-a-b}
U_g \coloneqq \Lambda(k_0)_{gN} = \Gamma(k_0)_{gN}  \subset A^{gN}. 
\end{equation}
\par 
Let $\varphi_g \colon A^{N} \to A^{gN} $ be the canonical bijection induced by the bijection $N \to gN$ given by $h \mapsto gh$. Then since $N$ is a memory set of $\sigma_t$,  we deduce from \eqref{e:pointwise-uniform-eventually-periodic-nuca=1-a} and \eqref{e:pointwise-uniform-eventually-periodic-nuca=1-a-b} that for every $g \in G \setminus EN^2$, the restriction map $t(g)\vert_{\varphi^{-1}(U_g)}\colon \varphi^{-1}(U_g) \to A$ is nothing but the restriction to  $\varphi^{-1}(U_g)\subset A^N$ of the  canonical projection $A^N \to A$ given by $ z \mapsto z(1_G)$. In other words, 
\begin{align}
\label{e:e:pointwise-uniform-eventually-periodic-nuca=2-a}
\sigma_t(x)\vert_{G \setminus EN^2} = x\vert_{G \setminus EN^2}, \quad \text{for all } x \in \Gamma(k_0).
\end{align}
\par 
 As $\sigma_t(\Gamma(k_0)) \subset \Gamma(k_0)$, an immediate induction using \eqref{e:e:pointwise-uniform-eventually-periodic-nuca=2-a}  shows that for every $k \geq 1$, we have for all $z \in \Gamma(k_0)$ that:  
\begin{equation}
\label{e:e:pointwise-uniform-eventually-periodic-nuca=2-b}
    \sigma_t^{k}(z)\vert_{G \setminus EN^2} = z\vert_{G \setminus EN^2} 
\end{equation}
\par 
Since $\sigma_s$ is pointwise eventually periodic, so is $\sigma_t$ by Lemma~\ref{l:simple-periodic-observation}. Hence, for every $w \in \Gamma(k_0)_{EN^3}\subset A^{EN^3}$, there exist integers $k_w, l_w \geq k_0$  such that 
\begin{equation}
\label{e:pointwise-even-per-main-relat-1-c-a}
\sigma_t^{k_w+l_w}(\tilde{w})= \sigma_t^{k_w}(\tilde{w})
\end{equation}
where $\tilde{w}\in \Gamma(k_0)$ is an arbitrary but fixed configuration  extending $w$. 
Let us define the positive  integers:  
\[
m=  \max \{k_w\colon w \in \Gamma(k_0)_{EN^3}\}, \quad    n=l_0 \prod_{w \in \Gamma(k_0)_{EN^3} }l_w, \quad n_w = \frac{n}{l_w}.   
\]
\par 
Let $x \in A^G$, $z= \sigma_s^{k_0}(x) \in  \Gamma(k_0)$, and 
$w = z\vert_{EN^3}$. Then we infer from \eqref{e:e:pointwise-uniform-eventually-periodic-nuca=2-b}  that 
\begin{equation}
\label{e:pointwise-eventually-periodic-nuca-1-a}\sigma_t^{n+m+k_0}(x)\vert_{G \setminus EN^2}= \sigma_t^{m+k_0}(x)\vert_{G \setminus EN^2}. 
\end{equation}
\par 
Since $N$ is a memory of $\sigma_t$ and $z\vert_{EN^3}=\tilde{w}\vert_{EN^3}$, we deduce by an immediate induction on $k$ that for every $k \geq 0$, we have: 
\begin{equation}
    \label{e:pointwise-even-per-main-relat}
\sigma_t^{k}(z)\vert_{EN^2}=\sigma_t^{k}(\tilde{w})\vert_{EN^2},  
\end{equation}
where the base case $k=0$ results  from \eqref{e:e:pointwise-uniform-eventually-periodic-nuca=2-b}.  
By the choice of $k_w, l_w$ and $m,n$ that for every $q \geq 1$  we have:  
 \begin{align}
 \label{e:pointwise-eventually-periodic-nuca-1-2-3}
  \sigma_t^{n+m+k_0}(x)\vert_{EN^2} &  =\sigma_t^{n+m} (\sigma_s^{k_0}(x)) \vert_{EN^2} \\
  & =\sigma_t^{n+m} (z) \vert_{EN^2}  \nonumber \\
  & =  \sigma_t^{n_wl_w+m}(\tilde{w})\vert_{EN^2} & \text{(by } \eqref{e:pointwise-even-per-main-relat}) \nonumber\\
  & = \sigma_t^{m}(\tilde{w})\vert_{EN^2} & \text{(by } \eqref{e:pointwise-even-per-main-relat-1-c-a} \nonumber \\
  & = \sigma_t^{m}(z)\vert_{EN^2} & \text{(by } \eqref{e:pointwise-even-per-main-relat}) \nonumber\\
  & =\sigma_t^{m+k_0}(x)\vert_{EN^2} \nonumber
 \end{align}
\par 
We can thus conclude from \eqref{e:pointwise-eventually-periodic-nuca-1-a} and \eqref{e:pointwise-eventually-periodic-nuca-1-2-3} that $\sigma_t^{n+m+k_0}(x)= \sigma_t^{m+k_0}(x)$ for all $x \in A^G$. Hence, $\sigma_t$ is eventually periodic and thus so is $\sigma_s$ by Lemma~\ref{l:simple-periodic-observation}. The proof is complete. 
\end{proof}

\section{Cayley-Hamilton linear NUCA}
\label{s:6}
\subsection{Linear NUCA satisfying polynomial equations}
Generalizing the nilpotency and eventual periodicity properties, we introduce the   Cayley-Hamilton property for linear NUCA which satisfy polynomial equations.  

\begin{definition}
\label{d:cayley-hamilton-lnuca} Let $\sigma_s \colon A^G \to A^G$ be a linear NUCA over a group universe $G$ and a $k$-vector space alphabet $V$. We say that $\sigma_s$ is \emph{pointwise Cayley-Hamilton} if for every $x\in A^G$, there exists a nonzero polynomial $P(z) \in k[z]$ depending on $x$ such that $P(\sigma_s)(x)=0$. We say that $\sigma_s$ is \emph{Cayley-Hamilton} if there exists a nonzero  polynomial $P(z) \in k[z]$ such that $P(\sigma_s)=0$.   
\end{definition}
\par 
It is clear from the above definition that every nilpotent linear NUCA and every periodic or eventually periodic linear NUCA are Cayley-Hamilton.
\par 
We begin with the following Lemma which generalizes Lemma~\ref{l:stable-periodic-nuca} in the case of linear NUCA with finite memory.  
\begin{lemma}
    \label{l:stable-cayley-hamilton-linear-nuca} 
Let $M $ be a finite subset of a  group $G$ and let $V$ be a $k$-vector space. Let $s \in \LL(V^M, V)^G$ and suppose that  there exists a polynomial $P(z) \in k[z]$ such that $P(\sigma_s)=0$. Then  $P(\sigma_t)=0$ for every $t \in \Sigma(s)$.   
\end{lemma}
\begin{proof}
Since $\sigma_s$ has finite memory, the proof is similar, \emph{mutatis mutandis}, to the proof of Lemma~\ref{l:stable-periodic-nuca}. 
\end{proof}
\par 
We establish the following general result which states that the pointwise Cayley-Hamilton property and the Cayley-Hamilton property are equivalent for linear NUCA with finite memory over countable vector space alphabet.  

\begin{theorem}
    \label{t:pointwise-cayley-hamilton-linear-nuca} 
Let $M $ be a finite subset of a countable group $G$ and let $V$ be a finite-dimensional vector space over a countable field $k$.   
Let $s\in \LL(V^M,V)^G$ and suppose that $\sigma_s$ is pointwise Cayley-Hamilton. Then $\sigma_s$ is Cayley-Hamilton, i.e., there exists a nonzero polynomial $P \in k[z]$ such that $P(\sigma_s)=0$. Moreover, for all $t \in \Sigma(s)$, we have $P(\sigma_t)=0$.  
\end{theorem}

\begin{proof}
Since the field $k$ is countable by hypothesis, so is  the polynomial rings $k[z]$. We can thus  enumerate  $k[z]= \{P_0 \colon n \in \N\}$ where $P_0=0$ is the trivial polynomial. Since $\sigma_s$ is continuous by \cite[Lemma~4.2]{phung-tcs}, we have for every $n \geq 1$ the following closed subset of $V^G$: 
\begin{align*}
    X_n = \Ker P_n(\sigma_s) =   \{ x \in V^G \colon P_n(\sigma_s)(x)=0 \}. 
\end{align*}
\par 
As $\sigma_s$ is pointwise Cayley-Hamilton, $V^G= \cup_{n \geq 1} X_n$ and thus the Baire category theorem implies that the interior of $X_{n_0}$ is nonempty for some $n_0 \geq 1$. Hence, there exist $E \subset G$ finite and $p \in V^E$ such that $W= \{p\}\times V^{G \setminus E} \in X_{n_0}$. Let $w_1, \dots, w_m$ be a basis of $V^{E}$. Then there exist $n_1, \dots , n_m \geq 1$ such that $P_{n_i}(\sigma_s)(\tilde{w}_i)=0$ for all $i=1, \dots, m$ where 
$\tilde{w}_i=w_i \times 0^{G \setminus E}$.  Consider the nonzero polynomial $P=\prod_{i=0}^n P_{n_i} \in k[z]$. Let $Q_i= \prod_{j  \neq i,  0 \leq j \leq m} P_{n_j} \in k[z]$ then observe for every $i=1, \dots, m$ that 
\begin{align*}
    P(\sigma_s)(w)= Q_i(\sigma_s) P_{n_i}(\sigma_s)(w_i)= Q_i(\sigma_s)(0) = 0. 
\end{align*}
\par 
 Let $x \in V^G$ and let $y=\{p\}\times x\vert_{G \setminus E} \in W$. Then for $z =x-y$, we have $z\vert_{G \setminus E}=0$ and we can find $a_1, \dots, a_m \in k$ such that $z =  a_1 \tilde{w}_1+ \dots + a_m \tilde{w}_m$. Hence, by linearity, we obtain: 
 \begin{align*}
     P(\sigma_s)(x) & = P(\sigma_s)(y) + P(\sigma_s)(z) \\
     & = Q_0(\sigma_s) P_{n_0}(\sigma_s)(y) + \sum_{i=1}^m a_i Q_i(\sigma_s)P_{n_i}(\sigma_s)(\tilde{w}_i) \\
     & = 0. 
 \end{align*} 
 \par 
 We conclude that $P(\sigma_s)=0$ and thus $\sigma_s$ is Cayley-Hamilton. Finally, the last statement results from Lemma~\ref{l:stable-cayley-hamilton-linear-nuca}. 
\end{proof}

\par 

The next theorem tells us that every Cayley-Hamilton linear NUCA over finite vector space alphabet  must be  eventually periodic, which greatly restricts the class of linear NUCA satisfying polynomial equations. 

\begin{theorem}
  \label{t:finite-cayley-hamilton-linear-ca-periodic} 
Let $G$ be a group and let $V$ be a finite vector space.   
Suppose that $\tau \colon V^G \to V^G$ is a pointwise Cayley-Hamilton linear NUCA. Then $\tau$ is eventually periodic. 
\end{theorem}

\begin{proof}
Let $k$ denote the base field of $V$. Note that $k$ is finite as $V$ is finite.  Since $\tau$ is pointwise Cayley-Hamilton, it is in fact Cayley-Hamilton by Theorem~\ref{t:pointwise-cayley-hamilton-linear-nuca}. Hence, there exists a polynomial $P(z) = z^n + a_{n-1}z^{n-1} + \dots + a_0 \in k[z]$ with $n \geq 1$ such that $P(\tau)=0$.   It follows that for every configuration $x \in  V^\Z$ and $y_k=\tau^k(x)$ where $k \geq 0$, we have:
\[
y_{k+n}+ a_{n-1}y_{k+n-1} + \dots + a_0y_k= 0 \quad \text{for all } k \geq 0. 
\]
\par 
Consequently, for every $g \in G$, the sequence have $(y_k(g))_{k \geq 0}$ is an $n$th-order linear recurrent sequence over $k$. Since the field $k$ is finite, it is well-known (see e.g. \cite{somer}) that there exists an integer $p\geq 1$ depending only on $k$ and $n$  such that $y_{k+p}(g)=y_k(g)$ for all 
$k \geq n-1$ and $g \in G$. 
\par 
Consequently,  $y_{k+p}= y_k$ for all $k \geq n-1$ and we thus conclude that 
$\tau^{p+n-1} (x)= \tau^{n-1}(x)$ for all $x \in V^G$. 
Thus, $\tau$ is eventually periodic and the proof is complete. 
\end{proof}

\subsection{Pointwise eventually periodic linear NUCA}
\par 
For linear NUCA with an arbitrary finite-dimensional vector space alphabet, it turns out that pointwise periodicity, resp. pointwise eventual periodicity, is also equivalent to periodicity, resp. eventual periodicity. 

\begin{theorem}
\label{t:pointwise-uniform-eventually-periodic-linear-nuca} 
Let $M$ be a finite subset of a countable group $G$ and let $V$ be a finite-dimensional vector space.  
Let $s\in \LL(V^M,V)^G$ and suppose that $\sigma_s$ is pointwise eventually periodic. Then  $\sigma^{n+m}_s = \sigma_s^m$ for some integers $m,n \geq 1$. Moreover, we have $\sigma_t^{n+m}= \sigma_t^m$ for all $t \in \Sigma(s)$.
\end{theorem}

\begin{proof}
    For integers $k, l \geq 1$, we define:
    \begin{align*}
X_{k,l} \coloneqq (\sigma_s^{k+l} - \sigma_s^k)^{-1}(0)=\Ker (\sigma_s^{k+l} -\sigma_k)  = \{ x \in V^G \colon \sigma_s^{k+l}(x) = \sigma_s^k(x)\}. 
    \end{align*}
\par 
$X_{k,l}$ is a closed subset of $V^G$ by the continuity of $\sigma_s^q$ for all $k \geq 1$ (cf.~\cite[Lemma~4.1]{phung-tcs}). Observe also that $X_{k,l} \subset X_{k,ql}$ for all $ k, l, q \geq 1$. 
Since $\sigma_s$ is pointwise eventually periodic by hypothesis, we have $V^G = \cup_{k, l \geq 1}^{\infty} X_{k,l}$. Hence, the Baire category theorem implies  that there exist $k_0, l_0 \geq 1$ such that the interior of $X_{m_0,n_0}$ is nonempty. Consequently, we can find a finite subset $E \subset G$ and a pattern $p \in V^E$ such that $W= \{p\} \times V^{G \setminus E} \subset X_{k_0, l_0}$. 
\par 
Let $v_1, \dots, v_m$ be a basis of the vector space $V^E$ where $m = \vert E \vert \dim V$. For every $i=1, \dots, m$, we define $w_i= v_i\times 0^{G \setminus E} \in V^G$. 
Since $\sigma_s$ is pointwise eventually periodic, there exist integers $k_i, l_i  \geq 1$ for $i=1, \dots, m$ such that $P_{i}(\sigma_s)(w_i)=0$ where $P_{i}(z)= z^{k_i}(z^{l_i}-1)$. 
Let us define $K= \max \{k_i\colon i=0, \dots, m\}$, $L=\prod_{i=0}^m l_i$, and $P(z)=z^{K} (z^{L}-1)$ then $P_i \vert P$ for all $i=0, \dots,m$. 
From this point, the same lines, \emph{mutatis mutandis}, of the proof of Theorem~\ref{t:pointwise-cayley-hamilton-linear-nuca} show that $P(\sigma_s)=0$. Thus $\sigma_s^{K+L}= \sigma_s^K$ and the last assertion follows from Lemma~\ref{l:stable-cayley-hamilton-linear-nuca}. The proof is complete. 
\end{proof}

For linear NUCA over an arbitrary finite-dimensional vector space alphabet, it turns out that pointwise periodicity is also equivalent to periodicity. 

\begin{theorem}
\label{t:pointwise-uniform-periodic-linear-nuca} 
Let $M$ be a finite subset of a countable group $G$ and let $V$ be a finite-dimensional vector space.  
Let $s\in \LL(V^M,V)^G$ and suppose that $\sigma_s$ is pointwise periodic. Then $\sigma_s$ is periodic, i.e.,  $\sigma^n_s = \Id$ for some integer $n \geq 1$. Moreover, $\sigma_t^n= \Id$ for every $t \in \Sigma(s)$.  
\end{theorem}

\begin{proof}
    For every integer $k \geq 1$, let us define:
    \[
X_k \coloneqq (\sigma_s^k - \Id)^{-1}(0)=\Ker ( \sigma_s^k -\Id)  = \{ x \in V^G \colon \sigma_s^k(x) = x\}. 
    \]
\par 
As $\sigma_s$ is pointwise periodic by hypothesis, we have $V^G = \cup_{k=1}^{\infty} X_k$. Then the Baire category theorem implies that there exist $n_0 \geq 1$, a finite subset $E \subset G$, and a pattern $p \in V^E$ such that $X_{n_0}$ contains $W= \{p\} \times V^{G \setminus E} \subset V^G$. 
\par 
Let $v_1, \dots, v_m$ be a basis of the vector space $V^E$ where $m = \vert E \vert \dim V$. For every $i=1, \dots, m$, we define $w_i= v_i\times 0^{G \setminus E} \in V^G$. 
Since $\sigma_s$ is pointwise periodic, there exist integers $n_1, \dots, n_m \geq 1$ such that $P_i(\sigma_s)(w_i) = 0$ where $P_i(z)=z^{n_i}-1$ for $i=0,1, \dots, m$. 
Let us define $n=\prod_{i=0}^m n_i\geq 1$ and $P(z)=z^n-1$. Then $P_i(z) \vert P(z)$ for all $i=0, \dots, m$ and we obtain as in the proof of Theorem~\ref{t:pointwise-cayley-hamilton-linear-nuca} and Theorem~\ref{t:pointwise-uniform-eventually-periodic-linear-nuca}  that $\sigma_s^n =\Id$. The last statement follows from  Lemma~\ref{l:stable-cayley-hamilton-linear-nuca}. The proof is thus complete. 
\end{proof}

\subsection{Pointwise nilpotent linear NUCA}

Let $X$ be a vector space, we say that a linear endomorphism $f \in \End(X)$ is \emph{pointwise nilpotent} if for every point $x \in X$, there exists an integer $n \geq 1$ such that $f^n(x)=0$. A linear endomorphism $f\in  \End(X)$ is said to be \emph{nilpotent} if there exists an integer $n \geq 1$ such that $f^n=0$.

\begin{theorem}
\label{t:pointwise-uniform-nilpotency-linear-nuca} 
Let $M $ be a finite subset of a countable group $G$ and let $V$ be a finite-dimensional vector space.  
Let $s\in \LL(V^M,V)^G$ and suppose that $\sigma_s$ is pointwise nilpotent. Then $\sigma_s$ is nilpotent, i.e., there exists an integer $n \geq 1$ such that $\sigma^n_s = 0$. Moreover, for every $t \in \Sigma(s)$, we have $\sigma_t^n= 0$. 
\end{theorem}

\begin{proof}
    For every $k \geq 1$, let us define:
    \[
X_k \coloneqq (\sigma_s^k)^{-1}(0)=\Ker \sigma_s^k  = \{ x \in V^G \colon \sigma_s^k(x) = 0\}. 
    \]
\par 
Since $\sigma_s$ is pointwise nilpotent by hypothesis, $V^G = \cup_{k=1}^{\infty} X_k$.  
Then as in the proof of Theorem~\ref{t:pointwise-uniform-nilpotency-nuca}, we can find a finite subset $E \subset G$ and a pattern $p \in V^E$ such that $X_{n_0}$, and thus every $X_k$ where $k \geq n_0$, contains the cylinder  $W= \{p\} \times V^{G \setminus E} \subset V^G$. 
\par 
Let us choose a basis $v_1, \dots, v_m$ of the finite-dimensional vector space $V^E$ where $m = \vert E \vert \dim V$. For every $i=1, \dots, m$, let $w_i= v_i\times 0^{G \setminus E} \in V^G$. 
Since $\sigma_s$ is pointwise nilpotent, there exist integers $n_1, \dots, n_m \geq 1$ such that $P_i(\sigma_s)(w_i) =0$ where $P_i(z)=z^{n_i}$ for $i=0,1, \dots, m$. Note that $P_i(\sigma_s)(w) =0$ for all $w \in W$. 
Let $P(z)=z^n$ where $n= \max \{n_0, n_1, \dots, n_m\} \geq 1$ then we deduce as in the proof of Theorem~\ref{t:pointwise-cayley-hamilton-linear-nuca} that $P(\sigma_s)=0$. Consequently, $\sigma_s^n=0$. The last statement is an immediate consequence of  Lemma~\ref{l:stable-cayley-hamilton-linear-nuca}.  
\end{proof}

\subsection{A generalization}

The above proofs of Theorem~\ref{t:pointwise-cayley-hamilton-linear-nuca} and Theorem~\ref{t:pointwise-uniform-eventually-periodic-linear-nuca}, Theorem~\ref{t:pointwise-uniform-periodic-linear-nuca}, and Theorem~\ref{t:pointwise-uniform-nilpotency-linear-nuca} can be adapted in a straightforward manner to obtain the following general result: 

\begin{theorem}
    \label{t:general-cayley-hamilton-sequence}
Let $M$ be a finite subset of a countable group $G$ and let $V$ be a finite-dimensional $k$-vector space.  Let $(P_n)_{n\geq 0}$ be a sequence of nonzero polynomials in $k[z]$ such that for every finite subset $S \subset \N$, there exists $m \geq 0$ such that $P_n \vert P_m$ for all $n \in S$. 
Let $s\in \LL(V^M,V)^G$ and suppose that for every $x \in V^M$, we have $P_n(\sigma_s)(x)=0$ for some $n \geq 0$. Then there exists $n \geq 0$ such that $P_n(\sigma_s) = 0$. Moreover, $P_n(\sigma_t)= 0$ for all $t \in \Sigma(s)$.  \qed 
\end{theorem}

\section{Some decidable properties of linear and group NUCA} 
\label{s:}
In this section, we shall show that the nilpotency, the periodicity, the eventual periodicity, the Cayley Hamilton property, and the injectivity of linear NUCA, and more generally of group NUCA, which are local perturbations of linear CA are decidable properties (cf. Theorem~\ref{t:decidable-lnuca-c-nilpotent}, Theorem~\ref{t:decidable-group-nuca-c-nilpotent},  Theorem~\ref{t:decidable-lnuca-c-periodic}, Corollary~\ref{c:decidable-group-nuca-c-cayley-hamilton}). Some generalizations to the case of sparse global perturbations of linear CA will be obtained in Section~\ref{s:}. 

\subsection{Nilpotency problem}

We first establish the following decidability result for the nilpotency of local perturbations of linear CA. 

\begin{theorem}
\label{t:decidable-lnuca-c-nilpotent}
    Let $G= \Z^d$ for some integer $d \geq 1$ and let $V$ be a finite vector space. Let $M \subset G$ be a finite subset and let $s \in \LL(V^M,V)^G$ be asymptotic to a constant configuration. Then it is decidable whether $\sigma_s$ is nilpotent. 
\end{theorem}

\begin{proof}
Up to enlarging $M$, we can suppose without loss of generality that $1_G \in M$ and $M=M^{-1}$. 
Since $\sigma_s \in \mathrm{LNUCA}_c(G,V)$  by hypothesis, there exist a constant configuration $c \in \LL(V^M,V)^G$ and a finite subset $E \subset G$ such that $s\vert_{G \setminus E} = c\vert_{G \setminus E}$. We claim that if $\sigma_s$ is nilpotent then so is $\sigma_c$. Indeed, suppose that $\sigma_s$ is nilpotent, then there exists an integer $n \geq 1$  such that 
    $\sigma_s^n(x)=0$ for all $x \in V^G$. By \cite[Theorem~6.2]{phung-tcs} (see also \cite[Theorem~A]{phung-twisted-group-ring}), there exists $t \in T^G$, where $T=\LL({V^N}, V)^G$  with $N=M^{n}$, such that $\sigma_t= \sigma_s^{n}$. Note that $\sigma_t \in \mathrm{LNUCA}_c(G,V)$ with finite memory $N=M^n$. Similarly, there exists a constant configuration $c' \in T^G$ such that $\sigma_c^n= \sigma_{c'}$. 
    Since $s\vert_{G \setminus E} = c\vert_{G \setminus E}$, we deduce that $c'(g) = t(g)$ for all $ g \in G \setminus EN$. Since $G= \Z^d$ is infinite, there exist $g_0 \in G \setminus EN$ and thus $c'(g) = c'(g_0)=t(g_0)$ for all $g \in G$ as $c'$ is constant. However, as $\sigma_t= \sigma_s^n=0$, the configuration $t$ is constant configuration $0^G$ and thus so is $c'$. Consequently, $\sigma_{c}^n= \sigma_{c'}=0$ and $\sigma_c$ is also nilpotent. The claim is thus proved. 
    \par 
    Now suppose that $\sigma_c$ is nilpotent. Hence, $\sigma_c^n=0$ for all large enough $n \geq n_0$ where $n_0 \geq 1$ is a fixed integer. Then as in the above paragraph, we can find $t \in T^G$, where $T=\LL({V^N}, V)^G$ and $N=M^{n_0}$, such that $\sigma_t= \sigma_s^{n_0}$. Since $\sigma_c^n=0$, we deduce that $t(g)=0$ for all $g \in G \setminus EN$.  
    \par 
    Consider the finite vector space  $W=\sigma_t(V^G)_{EN} = \im f^+_{EN, t\vert_{EN}} \subset V^{EN}$ and let $\varphi \colon W \to W$ be the  linear endomorphism defined as follows. For every $w \in W$, let $\tilde{w} = w \times 0^{G \setminus EN} \in V^G$ and we set: 
    \begin{equation}
    \label{e:decide-nil-lnuca-1}
         \varphi(w) = \sigma_t(\tilde{w})\vert_{EN}=f^+_{EN, t\vert_{EN}}(\tilde{w}\vert_{EN^2}) \in W.
    \end{equation}
    \par 
The map $\varphi$ is clearly linear as  $\sigma_t$ is linear. Moreover, since $t(g)=0$ for all $g \in G \setminus EN$,  we find for every integer $k \geq 0$ and every $x \in V^G$ that:
\begin{equation*}
\sigma_t^{k+1}(x)=\varphi^k(\sigma_t(x)\vert_{EN})\times 0^{G\setminus EN}. 
\end{equation*}
\par 
Consequently, the NUCA $\sigma_t$ is nilpotent if and only if $\varphi$ is nilpotent. To summarize, we infer from  Lemma~\ref{l:nilpotency-tau-tau-k}  that whenever $\sigma_c$ is nilpotent, the NCUA $\sigma_s$ is nilpotent if and only if so is $\varphi$. 
    \par
    With the above necessary and sufficient conditions for the nilpotency of $\sigma_s$ described above, we can derive an algorithm which decides whether $\sigma_s$ is nilpotent as follows. 
    Since $\sigma_c$ is a linear CA and $V$ is finite, \cite[Theorem~22]{kari-beaur} implies that we can decide whether $\sigma_c$ is nilpotent.  If $\sigma_c$ is not nilpotent, we conclude that $\sigma_s$ is not nilpotent. In the case $\sigma_c$ is nilpotent, we can easily compute an integer $n_0 \geq$ such that $\sigma_c^{n_0}=0$. We then compute $N=M^{n_0}$ and $t \in T^G$, where $T= \LL(V^N,V)$, such that $\sigma_t=\sigma_s^{n_0}$. Next we check the nilpotency of $\varphi \colon W \to W$ defined as above by computing $\varphi^{\dim W}$. If $\varphi^{\dim W}=0$ then we conclude that $\sigma_s$ is nilpotent. Otherwise, $\sigma_s$ is not nilpotent. The proof is thus complete. 
\end{proof}

\par 
By a similar argument, we obtain the following result which says that the nilpotency of group NUCA with finite memory and asymptotically constant configuration of local defining maps over the universe $\Z^d$ and a finite group alphabet is decidable. 

\begin{theorem}
\label{t:decidable-group-nuca-c-nilpotent}
    Let $G= \Z^d$ for some integer $d \geq 1$ and let $V$ be a finite group. Let $M \subset G$ be a finite subset and let $s \in \mathrm{Hom}_{grp}(V^M,V)^G$ be asymptotic to a constant configuration. Then it is decidable whether $\sigma_s$ is nilpotent. 
\end{theorem}

\begin{proof}
    The proof is the same, \emph{mutatis mutandis}, as the proof of Theorem~\ref{t:decidable-lnuca-c-nilpotent}. We only remark here that in order to check whether the group endomorphism $\varphi \colon W \to W$ is nilpotent, we can, for example,  simply check whether $\varphi^{m}(w)=e$ for all $w \in W$ where $e$ is the neutral element of $W$ and  $m = \vert V \vert^{\vert V \vert} \geq \vert \End_{grp}(V)\vert$ where   $\End_{grp}(V)$ denotes the monoid consisting of group endomorphisms of $V$.  
\end{proof}

\subsection{Periodicity and eventual periodicity problems}

By using the same strategy as in the proof of Theorem~\ref{t:decidable-lnuca-c-nilpotent}, we obtain the following decidability result for the periodicity and eventual periodicity  problem of local perturbations of linear CA.  
\begin{theorem}
\label{t:decidable-lnuca-c-periodic}
 Let $G= \Z^d$ for some integer $d \geq 1$ and let $V$ be a finite vector space. Let $M \subset G$ be a finite subset and let $s \in \LL(V^M,V)^G$ be asymptotic to a constant configuration. Then it is decidable whether $\sigma_s$ is periodic, resp.   eventually periodic. 
\end{theorem}

\begin{proof}
    Up to enlarging $M$, we can suppose without loss of generality that $1_G \in M$ and $M=M^{-1}$. 
Since $\sigma_s \in \mathrm{LNUCA}_c(G,V)$  by hypothesis, there exist a constant configuration $c \in \LL(V^M,V)^G$ and a finite subset $E \subset G$ such that $s\vert_{G \setminus E} = c\vert_{G \setminus E}$. As in the proof of Theorem~\ref{t:decidable-lnuca-c-nilpotent}, we find that  $\sigma_s$ is periodic, resp. eventually periodic,  only if so is $\sigma_c$. 
\par 
Suppose that $\sigma_c$ is periodic, resp. eventually periodic. Hence, we can find $n_0 \geq 1$ such that $\sigma_c^{n_0}=\Id$, resp. $\sigma_c^{2n_0}=\sigma_c^{n_0}$ (see Lemma~\ref{l:simple-tau-tau-k-per}).  Let $t \in T^G$, where $T=\LL({V^N}, V)^G$ and $N=M^{n_0}$, be such that $\sigma_t= \sigma_s^{n_0}$. 
For $g \in G$, let us consider $U_g=A^{N}$, resp. $U_g= \varphi_g^{-1}( \sigma_t(A^G)_{gN})$ where   $\varphi_g \colon A^{N} \to A^{gN} $ is the canonical bijection induced by the bijection $N \to gN$, $h \mapsto gh$.
Since $s\vert_{G \setminus E}=c\vert_{G \setminus E}$ and that $\sigma_c^{n_0}=\Id$, resp. $\sigma_c^{n_0}\vert_{\sigma_c^{n_0}(V^G)}=\Id\vert_{\sigma_c^{n_0}(V^G)}$, we deduce that for every $g \in G \setminus EN$, the map $t(g)\vert_{U_g}$, is exactly the restriction of the canonical projection $z \mapsto z(1_G)$ to $U_g$.   
\par 
Consequently, if we denote $W=V^{EN^2}$, resp. $W= \sigma_t(V^G)_{EN^2} \subset V^{EN^2}$, then we obtain a well-defined  linear endomorphism $\varphi \colon W \to W$ given  for every $w \in W$ by: 
\[
\varphi(w) = f^+_{EN, t\vert_{EN}}(w) \times w\vert_{EN^2\setminus EN}. 
\]
\par 
Moreover, we have $\sigma_t^k(x) = \varphi^k(x\vert_{EN^2})\times x\vert_{G \setminus EN^2}$ for every $x \in V^G$, resp. $x \in \im \sigma_t$. Therefore, $\sigma_t$ is periodic, resp. eventually periodic, if and only if so is the endomorphism $\varphi$. Hence, we infer from  Lemma~\ref{l:simple-tau-tau-k-per} that $\sigma_s$ is periodic, resp. eventually periodic, if and only if so is $\varphi$. 
\par
We deduce from the above the following  algorithm which decides whether $\sigma_s$ is periodic, resp. eventually periodic. 
By \cite[Theorem~22]{kari-beaur}, we can effectively decide whether $\sigma_c$ is periodic, resp. eventually periodic. If $\sigma_c$ is not periodic, resp. not eventually periodic, we conclude that $\sigma_s$ is not periodic, resp. not eventually periodic. Otherwise, we can effectively compute an integer $n_0 \geq$ such that $\sigma_c^{n_0}=0$, resp. $\sigma_c^{2n_0}=\sigma_c^{n_0}$. We then compute $N=M^{n_0}$ and $t \in T^G$, where $T= \LL(V^N,V)$, such that $\sigma_t=\sigma_s^{n_0}$. Let $W=\sigma_t(V^G)_{EN^2}= \im f^+_{EN, t\vert_{EN^2}}$ and define $\varphi \in \End(W)$ as above. Then since $\End(W)$ is finite as $W$ is finite, we can easily check whether $\varphi$ is periodic, resp. eventual periodic,  which tells us whether $\sigma_s$ is periodic, resp. eventual periodic. The proof is thus   complete. 
\end{proof}
\par 
As a trivial consequence of Theorem~\ref{t:finite-cayley-hamilton-linear-ca-periodic} and Theorem~\ref{t:decidable-lnuca-c-periodic}, we obtain: 
\begin{corollary}
\label{c:decidable-group-nuca-c-cayley-hamilton}
    Let $G= \Z^d$ for some integer $d \geq 1$ and let $V$ be a finite vector space. Let $M \subset G$ be a finite subset and let $s \in \LL(V^M,V)^G$ be asymptotic to a constant configuration. Then it is decidable whether $\sigma_s$ is Cayley-Hamilton. \qed
\end{corollary}
\par 
It is not hard to see that the above proof of Theorem~\ref{t:decidable-lnuca-c-periodic}  actually shows the following more general result for group NUCA with finite memory.  
\begin{theorem}
\label{t:decidable-group-nuca-c-periodic}
Let $G= \Z^d$ for some integer $d \geq 1$ and let $V$ be a finite group. Let $M \subset G$ be a finite subset and let $s \in \mathrm{Hom}_{grp}(V^M,V)^G$ be asymptotic to a constant configuration. Then it is decidable whether $\sigma_s$ is periodic, resp. eventually periodic.  \qed
\end{theorem}

\section{Some generalizations and applications} 
\label{s:8}
\subsection{Sparse global perturbations of NUCA}

If we allow multiple local perturbations of a configuration to happen remotely from each other then we obtain the following notion of sparse global perturbations. 

\begin{definition}
\label{d:sparse-perturbations}
    Let $M$ be a finite subset of a group $G$ and let $A$ be an alphabet. Let $S=A^{A^M}$ and let $s, c \in S^G$ be two  configurations. Then $s \in S^G$ is  called a \emph{sparse global perturbation} of $c$ if there exists a finite subset $E\subset G$ and a sequences $(g_n)_{n \geq 0}$ of elements in $G$ such that 
    \begin{enumerate}[\rm (i)]
        \item   $s\vert_{G \setminus F}= c\vert_{G \setminus F}$ where $F = \cup_{n \geq 0} g_n E$;
        \item 
        for every finite subset $D \subset G$, there exists an integer $n_0\geq 0$ such that $g_mD  \cap g_n D = \varnothing$         
        for all $m, n \geq n_0$. 
    \end{enumerate}
    \par 
    In this case, we say that $\sigma_c$ is a \emph{sparse global perturbation} of $\sigma_c$. 
\end{definition}
\par 
In particular, it is clear that if $s$ is asymptotic to $c$ then $s$ is a sparse global perturbation of $c$. 
\par 
\begin{example}
Let $G=\Z$ and let $f\colon \N \to \N$ be any  function such that $\lim_{n \to \infty}\frac{f(n)}{n}= \infty$. Let $g_n=\pm f(n)$ for $n \geq 0$. Then for every alphabet $A$, every finite subsets $M, E\subset G$, and configurations $s, c \in S^G$, where $S=A^{A^M}$, such that $s\vert_{G \setminus F}= c\vert_{G \setminus F}$, where $F = \cup_{n \geq 0} g_n E$, we see that $s$ is a sparse global perturbation of $c$. 
\end{example}

\subsection{Decidable properties of  perturbations of linear CA} 
Generalizing the results in Section~\ref{s:},  we show that various dynamical properties of linear NUCA which are sparse global perturbations of linear CA are also decidable.

\begin{theorem}
\label{t:decidable-sparse-perturbation-lnuca-nilpotent} 
 Let $G= \Z^d$, $d \geq 1$, and let $V$ be a finite vector space. Let $M \subset G$ be a finite subset and let $s \in \LL(V^M,V)^G$ be a sparse global perturbation of a constant configuration. Then it is decidable whether: 
 \begin{enumerate} [\rm (i)]
     \item  $\sigma_s$ is nilpotent; 
     \item  $\sigma_s$ is periodic;  
     \item  $\sigma_s$ is eventually periodic; 
     \item  $\sigma_s$ is Cayley-Hamilton. 
 \end{enumerate}
\end{theorem}

\begin{proof}
The proof is a direct generalization of the proofs of the main results of Section~\ref{s:}. We indicate below important modifications. First, we infer from  Theorem~\ref{t:finite-cayley-hamilton-linear-ca-periodic} that (iii)$\iff$(iv). By enlarging $M$, we can suppose that $1_G \in M$. By hypothesis, we are given a constant configuration $c \in  \LL(V^M,V)^G$, a finite subset $E\subset G$, and a sequence $(g_n)_{n \geq 0}$ in $G$ such that 
 \begin{enumerate}[\rm (a)]
        \item   $s\vert_{G \setminus F}= c\vert_{G \setminus F}$ where $F = \cup_{n \geq 0} g_n E$;
        \item 
        for every $D \subset G$ finite, there exists $n_0\geq 0$ such that $g_mD  \cap g_n D = \varnothing$         
        for all $m, n \geq n_0$. 
    \end{enumerate} 
\par 
Then a similar argument as in Theorem~\ref{t:decidable-lnuca-c-nilpotent} shows that if $\sigma_s$ satisfies $P$, where $P$ is any of the properties (i)-(iii), then so is $\sigma_c$. Note that $P$ is a decidable property for $\sigma_c$ by \cite[Theorem~22]{kari-beaur}. If $\sigma_c$ satisfies $P$ then we can compute a suitable integer $n_0 \geq 1$, $N=M^{n_0}$, and $t \in \LL(V^N, V)^G$ such that $\sigma_t=\sigma_s^{n_0}$ and $\sigma_c^{n_0}=0$, resp. $\sigma_c^{n_0}=\Id$, resp. $\sigma_c^{2n_0}=\sigma_c^{n_0}$ according to whether $P$ means nilpotency, resp. periodicity, resp. eventual periodicity. 
\par 
Then by (b), we can compute an integer $k_0 \geq 1$  such that $g_mEN^3 \cap g_n EN^3 = \varnothing$        and $g_mEN^3 \cap (\cup_{i=0}^{k_0} g_iEN^3)$ for all $m, n \geq k_0+1$. Then for each of the clusters $C_0=\cup_{i=0}^{k_0} g_iEN^3$, $C_1=g_{k_0+1}EN^3$, $C_2=g_{k_0+2}EN^3$, etc., we can effectively compute all the linear spaces $W_{0} \subset V^{EN^2}$, $W_1 \subset  V^{g_{k_0+1}EN^2}$, $W_2 \subset V^{g_{k_0+2}EN^2}$, etc. as well as the endormorphisms $\varphi_i\in \End(W_i)$ as in the proofs of Theorem~\ref{t:decidable-lnuca-c-nilpotent} and Theorem~\ref{t:decidable-lnuca-c-periodic}. Then it is clear that $\sigma_t$ and $\sigma_s$ satisfy $P$ if and only if so are the maps $\varphi_i$. Since the  spaces $V^{g_{k_0+i}EN^2}$ are finite and canonically isomorphic, there are only finitely many types of such maps $\varphi_i$. Therefore, we can effectively decide whether $\sigma_s$ satisfies $P$.  
\end{proof}

\section{Injectivity problem}
\label{s:9}
To generalize the known decidability result for the injectivity of higher-dimensional linear CA obtained in \cite{kari-beaur}, we need some preliminary results.  
As a straightforward extension of \cite{kitchen-schmidt} or \cite{phung-israel}, we have the following result: 

\begin{theorem}
\label{t:descend-stable-sft}
Let $A$ be a finite group and let $G=\prod_{1 \leq j \leq d} G_i$ be a monoid where $G_i= \N$ or $G_i=\Z$ for every  $i=1, \dots, d$. Then every closed group subshift $\Sigma \subset A^G$ is a subshift of finite type. Moreover, every descending chain of group subshifts of $A^G$ stabilizes. \qed   
\end{theorem}

\par 
Let $A$ be a set and let $G$ be a monoid, we say that a configuration $x \in A^G$ is \emph{eventually periodic} if its orbit under the shift action of $G$ is finite. More generally, if $G$ is a monoid of endomorphisms acting on a set $X$, then a point $x \in X$ is \emph{eventually periodic} if its orbit $Gx \subset X$ is finite.  
\par 
A straightforward and minor  modification of the proof of  \cite[Theorem~7.2]{kitchen-schmidt} actually shows the following. 

\begin{theorem}
\label{t:main-density}
Let $A$ be a finite abelian group and let $G=\prod_{1 \leq j \leq d } G_i$ be a monoid where $G_i= \N$ or $G_i=\Z$ for every  $i=1, \dots, d$. Then every closed group subshift  $\Sigma \subset A^G$ contains a dense subset of eventually periodic configurations.  \qed 
\end{theorem}
\par 
We only note below a few modifications needed in the original proof of \cite[Theorem~7.2]{kitchen-schmidt}. For instance, in \cite[Lemma~8.1]{kitchen-schmidt}, $\alpha$-periodic points should be modified to  $\alpha$-eventually periodic points. 
Let $e$ denote the neutral element of $A$. 
For every subgroup $H \subset A \times A$ and $Y(H) = \Sigma(G, A; H)$ where $G=\N$ or $G=\Z$, we define the following subgroups of $A$ (note the minor but important difference between the definition of $F_H(g,n)$ given below and the one in \cite{kitchen-schmidt} when $n \leq 0$): 
\[
F_H(g,n) 
= \begin{cases}
    \{y(n) \colon y \in Y(H), y(0)=e\} \quad & \text{if } n \geq 0;
\\
  \{y(0) \colon y \in Y(H), y(-n)=e\} \quad & \text{if } n \leq 0. 
\end{cases}
\]
\par 
Then $F_H(n) = F_H(e, n)$ is a closed normal subgroup of $G$ for every $n \in \Z$. Since $(e,e) \in H$, we have  $F_H(n+1) \subset F_H(n)$ and $F_H(-n) \subset F_H(-n-1)$ for all $n \geq 0$. Moreover, the kernel of the homomorphism  $G \to G / F_H(n)$, $g \mapsto F_H(g,n)/F_H(n)$ is equal to $F_H(-n)$. Hence, we also obtain an isomorphism: 
\[
\theta_n \colon G/F_H(-n) \to G / F_H(n). 
\]
and all other properties of $F_H(n)$. 
Let $\Lambda_1=F_H(-1) \cap F_H(1)$. 
Let $Y_n=Y_H \cap F_H(n)^G$. Then $Y_1\simeq \Lambda_1^G$ and, for example,  \cite[Lemma~5.7.(4)]{kitchen-schmidt} becomes 
\[
Y \simeq \Lambda_1^G \times \dots \times \Lambda_K^G\times G_K. 
\]
\par
We also need to modify the definition of $\mathrm{Fix}(\alpha^p)$ in \cite{kitchen-schmidt} to, for example, $\mathrm{Fix}(\alpha^p) = \{x \in X \colon \alpha^{2p}(x) = \alpha^p(x)\}$ to effectively deal with eventually periodic points instead of periodic points. 
\subsection{Effective constructions} 
Observe that not only the set of periodic configurations but the set of eventually periodic configurations are also countable. 
Therefore, using Theorem~\ref{t:main-density}, it is not hard to see that the results in \cite{kari-beaur} also hold for group subshifts and cellular automata over finite abelian group alphabets and monoid  universes of the form $\N^{d_1} \times \Z^{d_2}$. 
\par 
However, a few remarks are in order to interpret correctly their results in our slightly  more general setting.  Let $A$ be a finite abelian group and let $G= \prod_{1 \leq i \leq d} G_i$ where $G_i=\N$ or $\G_i=\Z$ for every $i=1, \dots, d$.
\begin{definition}
For every integer $n \geq 1$ and every $i=1, \dots, d$, the \emph{slice $\pi_i^{(n)}(X)$ of width $n$ in the $i$-th dimension} of a group subshift $X\subset A^{G}$ is defined as the projection $X_{G_1 \times \dots \times \{0,\dots, n-1\}\times \dots G_d  }$ regardless of whether $G_i=\N$ or $G_i=\Z$. 
\end{definition}
\par 
The notion of radius of synchronization should be modified similarly when we work with the monoid $\N$ instead of $\Z$. 
\par 
Let $\hat{G}_i= \prod_{1 \leq j \leq d, j\neq i } G_i$ for every $i=1, \dots, d$. For finite groups $A_1, A_2$ and a group subshift $X \subset (A_1 \times A_2)^G$, we denote the projection to the $i$-th alphabet $\Psi^{(i)} \colon (A_1 \times A_2)^G \to A_i^G$, $i=1,2$, by setting 
$\Psi^{(i)}(x)(g) = x^{(i)}(g)$ for all $g \in G$ and $x = (x^{(1)}, x^{(2)})\in (A_1 \times A_2)^G$ where $x^{(j)} \in  A_j^G$, $j=1,2$.  
More generally, let $E \subset F$ be finite sets and for each $i \in F$, let $A_i$ be a finite group. Then we define similarly the projection $\Psi^{E, F} \colon \prod_{i \in F}A_i^G \to \prod_{i \in E}  A_i^G$. 
\par 
From the proof of \cite[Lemma~9]{kari-beaur}, we can deduce  the following result: 
\begin{theorem}
\label{t:main-effective-construction}
Let $A, A_1, A_2$ be finite abelian groups and let $G=\prod_{1 \leq j \leq d} G_i$ be a monoid where $G_i= \N$ or $G_i=\Z$ for every  $i=1, \dots, d$. Then the following hold: 
\begin{enumerate}[\rm (a)]
    \item given a closed group subshift $X \subset A^G$, we can effectively construct the $d-1$-dimensional group subshift  $\pi_i^{(n)}(X) \subset (A^{n})^{\hat{G}_i}$ of width $n$ in the $i$-th dimension for every $n \geq 1$ and $i=1, \dots, d$; 
    \item 
    given a closed group subshift  $X \subset (A_1 \times A_2)^G$, we can effectively construct the $d$-dimensional group subshift $\Psi^{(1)}(X) \subset A_1^G$.  
\end{enumerate}
\end{theorem}

\begin{proof}
We indicate only the following nontrivial modification in the proof of (a) for dimension $d$ with $G_i=\N$ assuming (b) holds for dimension $d-1$. 
Let $P$ be a finite set of forbidden patterns of $X$ so that we can denote $X=\mathcal{X}_P$ (see the proof of \cite[Lemma~9]{kari-beaur} for more details on the notations). We can  suppose that $n$ is larger than  the width of all the forbidden patterns in $P$. 
\par 
Instead of constructing directly as in the proof of \cite[Lemma~9]{kari-beaur} the upper approximation group subshifts  $\mathcal{X}_Q$ of $Y= \pi_i^{(n)}(X) \subset (A^{n})^{\hat{G}_i}$, we need to proceed more carefully as follows. 
For every $m \geq n+1$, we denote 
$E_m=\{m-n,\dots, m-1\}$, 
$E'_m= \{m-n-1,\dots, m-2\}$,  $F_m= \{0,1,\dots,m -1\}$,  and consider the following group shifts:
\begin{align*}
    Z_m = \pi_i^{(m)}(X)  \subset & (A^{F_m})^{\hat{G}_i}=(A^m)^{\hat{G}_i}, \\
    L_m = \Psi^{E'_m, F_m}(Z_m)  \subset (A^{n})^{\hat{G}_i} &, \quad  R_m = \Psi^{E_m, F_m}(Z_m)  \subset  (A^n)^{\hat{G}_i}. 
\end{align*} 
\par 
Then $ L_{m+1} = R_{m} \subset L_m$ and we deduce from Theorem~\ref{t:descend-stable-sft} that $R_{m}=L_m$ for all large enough $m$.  
\par 
Now we proceed as in the proof of \cite[Lemma~9]{kari-beaur} to obtain a decreasing sequence of upper approximation  group shifts  $\mathcal{X}_{Q_{m,k}} \subset (A^{F_m})^{\hat{G}_i}$ of $Z_m$ for each $m \geq n +1$ by testing and adding suitably  more and more forbidden patterns of $Z_m$ to obtain $Q_{m,k}$. Let us define
\begin{align*}
    T^L_{m,k}= \Psi^{E'_m, F_m}(\mathcal{X}_{Q_{m,k}}) \subset (A^n)^{\hat{G}_i}, 
    \quad T^R_{m,k}= \Psi^{E_m, F_m}(\mathcal{X}_{Q_{m,k}}) \subset (A^n)^{\hat{G}_i}.  
\end{align*}
\par 
Then $T^L_{m,k}$ and $T^R_{m,k}$ can be effectively constructed by our assumption that (b) holds for dimension $d-1$. It is clear as in the proof of Claim~1 in  \cite[Lemma~9]{kari-beaur} that $\mathcal{X}_{Q_{m,k}}=Z_m$ if and only if $T^{L}_{m,k}=T^R_{m,k}$. 
By Theorem~\ref{t:descend-stable-sft} and the above discussion on $Z_m$, such a pair $(m_0,k_0)$ exists and thus can be effectively computed. Finally, it suffices to compute $\pi_i^{(n)}(X) \subset (A^{n})^{\hat{G}_i}$  by (b) as the projection: 
\begin{align*}
\pi_i^{(n)}(X)  = \Psi^{\{0,\dots, n\}, F_{m_0}}(Z_m) 
=  \Psi^{\{0,\dots, n\}, F_{m_0}}(\mathcal{X}_{Q_{m_0,k_0}}). 
\end{align*}
\par 
The rest of the proof follows the same proof, \emph{mutatis mutandis}, of \cite[Lemma~9]{kari-beaur}. 
\end{proof}
\subsection{Injectivity is decidable for $\mathrm{LNUCA}_c$} 
We can now state and prove that the injectivity is   a decidable  property for local perturbations of higher dimensional linear CA.
\begin{theorem}
    \label{t:main-injectivity}
  Let $G= \Z^d$ for some integer $d \geq 1$ and let $V$ be a finite vector space. Let $M \subset G$ be a finite subset and let $s \in \LL(V^M,V)^G$ be asymptotic to a constant configuration. Then it is decidable whether $\sigma_s$ is injective. 
\end{theorem}

\begin{proof}
We can suppose without loss of generality that $1_G \in M$. 
By hypothesis, we are given a constant configuration $c \in \LL(V^M,V)^G$ and a finite subset $E \subset G$ such that $s\vert_{G \setminus E}=c\vert_{G \setminus E}$. 
\par 
Since $G=\Z^d$ is an amenable group, we infer from \cite[Theorem~A]{phung-tcs} and \cite[Theorem~B]{phung-tcs} that $\sigma_s$ is injective if and only if it is invertible if and only if we can find a finite finite subset $N \subset G$ with $1_G \in N$ and some $t \in \LL(V^N,V)^G$ asymptotically constant such that $\sigma_t \circ \sigma_s = \Id$ or equivalently, for every $g \in EN^2M^2$, the map 
\[
\varphi_{g}(t)=f^+_{gN, t\vert_{gN}}\circ f^+_{gNM, s\vert_{gNM}}
\] 
is equal to the canonical projection $V^{gNM} \to V^{\{g\}}$.  
Consequently, in the case $\sigma_s$ is injective, we can simply enumerate all finite subsets $N \subset G$ and all  asymptotically constant configurations $t \in \LL(V^N,V)^G$ and test whether the above finitely many maps $\varphi_{g}(t)$ is the canonical projection $V^{gNM} \to V^{\{g\}}$. 
\par 
In parallel, to deal with the case when $\sigma_s$ is not injective, we proceed as follows. To simplify the notations, we present the proof in the case $d=2$. The general case is similar. 
We can suppose that $M, E \subset \{-r, \dots, r\}^2 \subset \Z^2$ for some integer $r \geq 1$ which will be incremented at each step of our algorithm.  
\par 
For every $a \in \Z$, we denote $\Z_{\geq a} = \{n \colon n \geq a \}\subset \Z$ and similarly 
$\Z_{\leq a} = \{n \colon n \leq a \}$. 
Let 
$T^L \subset V^{\Z_{\leq -3r}\times \Z}$ and $T^R \subset V^{\Z_{\geq 3r}\times \Z}$ be the group shifts of finite type defined respectively as the kernels of linear maps 
\[
\sigma_c\vert_{ V^{\Z_{\leq -3r}\times \Z}} \colon  V^{\Z_{\leq -3r}\times \Z} \to  V^{\Z_{\leq -2r}\times \Z}, \quad 
\sigma_c\vert_{ V^{\Z_{\leq 3r}\times \Z}} \colon  V^{\Z_{\leq 3r}\times \Z} \to  V^{\Z_{\leq 2r}\times \Z}. 
\] 
\par 
Then by Theorem~\ref{t:main-effective-construction}, we can  compute the slices 
$S^L \subset V^{\{-5r,\dots,  -3r\}\times \Z}$ and 
$S^R \subset V^{\{3r,\dots,  5r\}\times \Z}$ of $T^L$ and $T^R$ respectively.
We can respectively compute  finite sets $P_L$ and $P_R$ of forbidden patterns of the group shifts of finite type $S^L$ and $S^R$. Let $h_0\geq 0$ be the maximal  height of the patterns in $P_L$ and $P_R$.  
\par 
Similarly, for every integer $a \geq 0$,  let $T^U(a) \subset  V^{\{-5r,\dots,  5r\}\times \Z_{\geq a}}
$ and $T^D(a)\subset  V^{\{-5r,\dots,  5r\}\times \Z_{\leq -a}}$ be respectively the kernels of the linear maps: 
\begin{align*}
\sigma_c\vert_{V^{\{-5r,\dots,  5r\}\times \Z_{\geq a}}} & \colon  V^{\{-5r,\dots,  5r\}\times \Z_{\geq a}} \to  V^{\{-4r,\dots, 4r\}\times \Z_{\geq a+r}}
, \\  
\sigma_c\vert_{V^{\{-5r,\dots,  5r\}\times \Z_{\leq - a}}} & \colon  V^{\{-5r,\dots,  5r\}\times \Z_{\leq -a}} \to  V^{\{-4r,\dots, 4r\}\times \Z_{\leq -a-r}}. 
\end{align*}
\par 
Hence, by adding the forbidden patterns $P_L \cup P_R$ of both $S^L$ and $S^R$ to the set of forbidden patterns of $T^U$ and $T^D$ respectively, we obtain the following effectively computable group shifts of finite type  for every integer $a \geq 0$:  
\begin{align*}
    W^U(a) & \subset  V^{\{-5r,\dots,  -5r\}\times \Z_{\geq a}}, \\ W^D(a) & \subset V^{\{-5r,\dots,  -5r\}\times \Z_{\leq -a}}. 
\end{align*}
\par 
Again by Theorem~\ref{t:main-effective-construction}, we can effectively compute the slices 
$S^U(a) \subset V^{\{-5r,\dots,  -5r\}\times \{a,\dots,  a+h_0+2r\}}$ and 
$S^D(a) \subset V^{\{-5r,\dots,  -5r\}\times \{-a,\dots,  -a-h_0-2r\}}$ of the group shifts  $W^U(a)$ and $W^D(a)$ respectively.

\par 
Finally, we have gathered all the ingredients to describe an algorithm which detects the existence of a nonzero configuration in the kernel of $\sigma_s$ as follows. Let us define $b=2r+h_0$, $F= \{-4r, \dots , 4r\} \times \{-b, \dots, b\}$ then  
$\Ker (f^+_{F, s\vert_F}) \subset V^Z$ where $Z=\{-5r, \dots , 5r\} \times \{-3r-h_0, 3r+h_0\}$. Let us consider the following finite sets 
\begin{align*}
    Z^U & ={\{-5r, \dots , 5r\} \times \{r, \dots, 3r+h_0\}},
    \\ Z^D& ={\{-5r, \dots , 5r\} \times \{-r, \dots, -3r-h_0\}}\\  
    Z^L& =\{-5r, \dots , -3r\} \times \{-3r-h_0, 3r+h_0\}, \\ Z^R& =\{3r, \dots , 5r\} \times \{-3r-h_0, 3r+h_0\}. 
\end{align*}
\par 
We define the following effectively computable set  $Y=\Ker (f^+_{F, s\vert_F}) \cap Y_U \cap Y_D \cap Y_L \cap Y_R \subset V^Z$  where:  
\begin{align*}
Y^U =\{y \in Z \colon y\vert_{Z^U} \in S^U(r)\}&, \quad 
Y^D = \{y \in Z \colon y\vert_{Z^D} \in S^D(r)\},
\\ 
Y^L= \{y \in Z \colon y\vert_{Z^L} \in T^L_{Z^L}\}&, \quad 
Y^R= \{y \in Z \colon y\vert_{Z^R} \in T^R_{Z^R}\}. 
\end{align*}
\par 
Then it is clear by our constructions that every element $y \in Y$ can be extended to an element in $\Ker(\sigma_s)$ and conversely we have $x\vert_Z \in Y$ for  every $x \in \Ker(\sigma_s)$. In other words, 
$(\Ker \sigma_s)_{Z} = Y$. Consequently, if $\sigma_s$ is not injective then for large enough $r$, we have $Y \neq \{0^Z\}$, which is trivially a decidable property, and we can conclude that $\sigma_s$ is not injective. In the case $Y=\{0^Z\}$, we simply increase $r$ and repeat the above procedure. The proof is thus complete. 
\end{proof}

\par 
\begin{proof}[Proof of Theorem~\ref{t:intro-decidable-local-perturbation-lnuca}]
Assertions (i)-(iv) result directly from Theorem~\ref{t:decidable-sparse-perturbation-lnuca-nilpotent} and Theorem~\ref{t:main-injectivity}. Since the  post-surjectivity of $\sigma_s$ is equivalent to the effectively computable dual linear NUCA $\sigma_s^*$ by \cite{phung-dual-nuca}, Theorem~\ref{t:intro-decidable-local-perturbation-lnuca}.(v) is a consequence of (iv). 
\end{proof}

\par 
Finally, we note that our proofs, \emph{mutatis mutandis}, actually show the following result which is slightly more general than Theorem~\ref{t:intro-decidable-local-perturbation-lnuca} for most properties. 
\par 
\begin{theorem}
\label{t:general-decidable-local-perturbation-lnuca} 
Let $G= \Z^d$, $d \geq 1$, and let $V$ be a finite vector space. Let $M \subset G$ be a finite subset and let $S= V^{V^M}$. Let $s \in S^G$ be asymptotic to a constant configuration $c \in \LL(V^M,V)^G$. Then it is decidable whether: 
 \begin{enumerate} [\rm (i)]
     \item  $\sigma_s$ is nilpotent; 
     \item  $\sigma_s$ is (eventually) periodic;  
     \item  $\sigma_s$ is Cayley-Hamilton. 
     \item 
     $\sigma_s$ is injective.  \qed 
 \end{enumerate}

\end{theorem}

\bibliographystyle{siam}

\begin{thebibliography}{10}


  
\bibitem{hedlund-csc}
{\sc T.~Ceccherini-Silberstein, and M.~Coornaert}, 
A generalization of the Curtis-Hedlund theorem, Theoret. Comput. Sci., 400 (2008), pp. 225–229
 
 
 

\bibitem{cscp-alg-goe}
{\sc T.~Ceccherini-Silberstein, M.~Coornaert, and X.~K. Phung},  
{\em On the Garden of Eden theorem for endomorphisms of symbolic algebraic varieties}, 
 Pacific J. Math. 306 (2020), no. 1, pp~31--66. 



\bibitem{Den-12a}  
{\sc A. Dennunzio, E. Formenti, and J. Provillard},  {\em Non-uniform cellular automata: Classes, dynamics, and decidability},  
Information and Computation
Volume 215, June 2012, Pages 32-46
 
\bibitem{Den-12b} 
{\sc A. Dennunzio, E. Formenti, and J. Provillard},   {\em Local rule distributions, language complexity and non-uniform
cellular automata}. Theoretical Computer Science, 504  (2013) 38–51.  

\bibitem{formenti-hoca}
{\sc A. Dennunzio, E. Formenti, L. Manzoni, L. Margara, and A. E. Porreca}, 
{\em On the dynamical behaviour of linear higher-order cellular automata and its decidability}. Information Sciences (2019), 486, pp.~73--87. 

\bibitem{formenti-ergodicity}
{\sc Dennunzio, E. Formenti, D. Grinberg, L.~Margara}, 
{\em Chaos and ergodicity are decidable for linear cellular automata
over}, Information Sciences (2020), 539, pp.~136--144. 

  
 \bibitem{GOL} 
{\sc M.~Gardner}. 
{\em Mathematical Games: The Fantastic Combinations of John Conway’s New Solitaire Game “Life”}. Scientific American, 223 (4), 1970,  pp.~120--123.

 


\bibitem{hedlund}  
{\sc G. A. Hedlund}, 
{\em Endomorphisms and automorphisms of the shift dynamical system}, 
Math. Systems Theory, 3 (1969), pp.~320--375.


\bibitem{kari-nilpotent}
{\sc J. Kari}
{\em The nilpotency problem of one-dimensional cellular automata}. SIAM Journal on Computing (1992), 21(3), pp.~571--586.

\bibitem{kari-reversible}
\leavevmode\vrule height 2pt depth -1.6pt width 23pt,
{\em  Reversibility and surjectivity problems of cellular automata}. Journal of Computer and System Sciences (1994), 48(1), pp.~149--182.  

\bibitem{kari-beaur}
{\sc P. Béaur and J. Kari}, 
{\em Decidability in Group Shifts and Group Cellular Automata}, 45th International Symposium on Mathematical Foundations of Computer Science (MFCS 2020), vol.~170, pp.~12:1--12:13. doi: 10.4230/LIPIcs.MFCS.2020.12


\bibitem{kitchen-schmidt} 
{\sc B. Kitchens and K. Schmidt}, 
{\em Automorphisms of compact groups}, Ergodic Theory and Dynamical Systems (1989) vol. 9, pp.~691--735.




\bibitem{neumann}
{\sc J.~von Neumann}, 
{\em The general and logical theory of automata}, Cerebral Mechanisms in Behavior. The Hixon Symposium, John Wiley \& Sons Inc., New York, N. Y., 1951, pp.~1--31; discussion, pp. 32--41.


\bibitem{phung-2020}
{\sc X.K.~Phung}, 
{\em On sofic groups, Kaplansky's conjectures, and endomorphisms of pro-algebraic groups}, Journal of Algebra, 562 (2020), pp.~537--586. 


\bibitem{phung-shadowing}
\leavevmode\vrule height 2pt depth -1.6pt width 23pt,
{\em 
Shadowing for families of endomorphisms of generalized group shifts}, Discrete and  Continuous Dynamical Systems, 2022, 42 (1) : 285-299

\bibitem{phung-embedding}
\leavevmode\vrule height 2pt depth -1.6pt width 23pt,
{\em On images of subshifts under embeddings of symbolic varieties}, Ergodic Theory and Dynamical Systems, pp.~1--19. https://doi.org/10.1017/etds.2022.48

\bibitem{phung-israel}
\leavevmode\vrule height 2pt depth -1.6pt width 23pt,
{\em On dynamical finiteness properties of algebraic group shifts}, 
Israel Journal of Mathematics (2022). https://doi.org/10.1007/s11856-022-2351-1


\bibitem{phung-tcs}
\leavevmode\vrule height 2pt depth -1.6pt width 23pt, 
{\em On invertible and stably reversible non-uniform cellular automata},  Theoretical Computer Science (2022). https://doi.org/10.1016/j.tcs.2022.09.011 


\bibitem{phung-dual-nuca} 
\leavevmode\vrule height 2pt depth -1.6pt width 23pt, 
{\em On linear non-uniform cellular automata: duality and dynamics}, preprint. arXiv:2208.13069 

\bibitem{phung-twisted-group-ring}
\leavevmode\vrule height 2pt depth -1.6pt width 23pt, 
{\em Stable finiteness of twisted group rings and noisy linear cellular automata}, preprint. arXiv:2209.06002


\bibitem{somer}
{\sc L. Somer}, 
{\em Periodicity Properties of $k$-th order  Linear Recurrences with Irreducible 
Characteristic Polynomial over a Finite Field},  Finite Fields, Coding Theory and Advances in Communications and Computing. Edited by Gary Mullen and Peter Jau-Shyong
Shiue, Marcell Dekker Inc., 1993, pp.~195--207

  
\end{thebibliography}

\end{document}